\documentclass[12pt]{article}

\usepackage{amsfonts,amsmath,amsthm,amssymb,amscd} 
\usepackage[dvips]{graphicx}
\usepackage{mathrsfs}

\theoremstyle{definition}
\newtheorem{theorem}{Theorem}
\newtheorem{lemma}[theorem]{Lemma}
\newtheorem{proposition}[theorem]{Proposition}
\newtheorem{corollary}[theorem]{Corollary}

\numberwithin{equation}{section}
\numberwithin{theorem}{section}

\begin{document}

\begin{center}
{\bf{\Large Fuchsian differential equations with modular forms  }}
\end{center}

\begin{center}
By Kazuhide Matsuda
\end{center}

\begin{center}
Faculty of Fundamental Science,  \\
National Institute of Technology (KOSEN),  Niihama College,\\
7-1 Yagumo-chou, Niihama, Ehime, Japan, 792-8580. \\
E-mail: matsuda@sci.niihama-nct.ac.jp  \\
Fax: 81-0897-37-7809 
\end{center}

\noindent
{\bf Abstract}
The aim of this paper is to derive new results about Jacobi's inversion formulas for modular forms of levels 5 and 6. 
For this purpose, we use Farkas and Kra's theory of theta functions with rational characteristics. 

\section{Introduction}
There have been many researches about relationship between linear ordinary differential equations (ODEs) and modular forms. 
In particular, Jacobi proved that the complete elliptic integral can be expressed by the theta constant:
\begin{equation}
\label{eqn:Jacobi-inversion-level-2}
K=
\int_0^1 
\frac{dt}{\sqrt{(1-t^2)(1-xt^2)   } } =\frac{\pi}{2}\theta_3^2,
\end{equation}
where $K$ is a solution of the ODE, 
\begin{equation*}
x(1-x)
\frac{d^2 y}{d x^2}
+(1-2x)
\frac{dy}{d x}
-\frac14
y=0. 
\end{equation*}
\par
In his second notebook, Ramanujan made a different and concise approach to {\it Jacobi's inversion formula} (\ref{eqn:Jacobi-inversion-level-2}).  
Following Ramanujan, Berndt \cite[pp. 119]{Berndt-1} proved that
\begin{equation*}
\left( \sum_{n=-\infty}^{\infty} q^{n^2} \right)^2
=
{ }_2 F_1 
\left(
\frac12, \frac12; 1; x
\right),  \,\,
x=16 q \prod_{n=1}^{\infty} \frac{ (1+q^{2n})^8 }{(1+q^{2n-1})^8  },
\end{equation*}
where ${ }_2 F_1(a,b;c;z)$ is {\it the Gaussian hypergeometric function} defined by 
\begin{equation*}
{ }_2 F_1(a,b;c;z)
=
\sum_{n=0}^{\infty}
\frac{ (a)_n (b)_n  }{ (c)_n n!  } z^n, \,\,
(a)_0=1 \,\,
\mathrm{and} \,\,
(a)_n=a(a+1)(a+2)\cdots (a+n-1).
\end{equation*}
It is known that ${ }_2 F_1(a,b;c;z)$ satisfies {\it the Gauss hypergeometric ODE},
\begin{equation*}
x(1-x)\frac{d^2 y}{d x^2}
+
\left\{
c-(a+b+1)x
\right\}
\frac{dy}{d x}
-ab y
=0. 
\end{equation*}
Jacobi's inversion formula for modular forms of level 3 was given by the Borweins \cite{Borweins} and Berndt et al. \cite{Berndt-Bhargava-Garvan}. 
\par
Beukers \cite{Beukers} derived Jacobi's inversion formula for modular forms of level 5 in order to prove Ap\'ery's work which state that 
the value $\zeta(3)$ of the Rieman's zeta-function is irrational. 
\par
Inspired by Beukers' work, Zagier was concerned with the sequence $\left\{ t(n) \right\} \,\,(n\geq 0) $ defined by the recurrence formula, 
\begin{equation}
\label{eqn:Zagier-sporadic}
(n+1)^2t(n+1)=(\alpha n^2+\alpha n+\beta)t(n)+\gamma n^2 t(n-1), t(0)=1,
\end{equation}
where $\alpha, \beta$ and $\gamma$ are integers. 
He studied whether the recurrence formula (\ref{eqn:Zagier-sporadic}) takes only integer values, and 
obtained {\it Zagier's sporadic sequence}.
\par
Cooper \cite{Cooper} used Ramanujan's theory of theta functions to systematically derive Jacobi's inversion formulas for modular forms of levels 1-12. 
In addition, he considered the relationship between the formulas and Zagier's sporadic sequences. 
\par
The aim of this paper is to derive new results about Jacobi's inversion formulas for modular forms of levels 5 and 6. 
For this purpose, we first use Farkas and Kra's theory of theta functions with rational characteristics to obtain nonlinear differential equations satisfied by modular forms. 
We next derive Jacobi's inversion formulas from the nonlinear ODEs.

\section{Notations}
\label{sec:notations}
For the positive integers $j,k,$ and $n\in\mathbb{N}$ 
$d_{j,k}(n)$ denotes the number of positive divisors $d$ of $n$ such that $d\equiv j \,\,\mathrm{mod} \,k,$ 
and 
$d_{j,k}^{*}(n)$ denotes the number of positive divisors $d$ of $n$ such that $d\equiv j \,\,\mathrm{mod} \,k$ and $n/d$ is odd,   
which implies that 
$
d_{j,k}^{*}(n)=d_{j,k}(n)-d_{j,k}(n/2). 
$ 
Moreover, 
for $k$ and $n\in\mathbb{N}$ 
the arithmetical function $\sigma_k(n)$ is the sum of the $k$-th power of the positive divisors of $n,$ 
and 
$d_{j,k}(n)=\sigma_k(n)=0$ for $n\in\mathbb{Q}\setminus\mathbb{N}_0.$ 
\par
For a function $y=y(x)$ the Schwartzian derivative is defined by 
\begin{equation*}
\{y, x \}
=
\frac{y^{\prime \prime \prime}}{y^{\prime}}
-\frac32
\left(
\frac{y^{\prime \prime }}{y^{\prime}}
\right)^2. 
\end{equation*}
Throughout this paper
the {\it upper half plane} is defined by
$
\mathbb{H}^2=
\{
\tau\in\mathbb{C} \,\, | \,\, \Im \tau>0
\}.
$
We set $q=\exp(2\pi i  \tau)$ and write 
\begin{align*}
&\varphi(q)=\sum_{n\in\mathbb{Z}} q^{n^2}, \,\,
\psi(q)=\sum_{n=0}^{\infty}  q^{\frac{n(n+1)}{2}},  \,\,
a(q)=a(\tau)
=\sum_{m,n\in\mathbb{Z}} q^{m^2+mn+n^2},  \\
&
b(q)=b(\tau)
=\sum_{m,n\in\mathbb{Z}} \omega^{n-m} q^{m^2+mn+n^2},  \,\,
c(q)=c(\tau)=\sum_{m,n\in\mathbb{Z}} q^{(n+\frac13)^2+(n+\frac13)(m+\frac13)+(m+\frac13)^2}, \,\,\omega=e^{2\pi i/3}.
\end{align*}
Klein's $J$-function is defined by 
\begin{equation*}
J(\tau)
=\frac{j(\tau)}{12^3}
=
\frac{  E_4(\tau)^3  }{E_4(\tau)^3-E_6(\tau)^2 }, 
\end{equation*}
where 
\begin{align*}
E_2(\tau)=&1-24\sum_{n=1}^{\infty} \sigma_1(n) q^n, \,\,
E_4(\tau)=1+240\sum_{n=1}^{\infty} \sigma_3(n) q^n, \\
E_6(\tau)=&1-504\sum_{n=1}^{\infty} \sigma_5(n) q^n, \,\, \sigma_k(n)=\sum_{d|n} d^k \,\,(k=1,2,3,4,5,\ldots),   \,\,
q=\exp(2 \pi i \tau). 
\end{align*}

\begin{lemma}
\label{lem:hypergeometric-limit-1-divergence}
{\it
Suppose that $\gamma-\alpha-\beta=0.$ 
Then 
\begin{equation*}
\lim_{x\to 1-0} 
\frac{
{  }_2 F_1(\alpha, \beta  ; \gamma; x)
}
{
\frac
{
\Gamma(\alpha+\beta)
}
{
\Gamma(\alpha) \Gamma(\beta)
}
\log
\frac{1}{1-x}
}
=1
\end{equation*}
}
\end{lemma}

\begin{proof}
For the proof, see Whittaker and Watson \cite[p. 299]{WW}. 
\end{proof}

\section{The Hypergeometric equations and the Eisenstein series}
\label{sec:hypergeometric-E_4-E_6}
We recall Ramanujan's system of ODEs, 

\begin{align}
DE_2(\tau)=&\frac{E_2(\tau)^2-E_4(\tau)  }{12  }, \,\,
DE_4(\tau)=\frac{E_2(\tau) E_4(\tau)-E_6(\tau)  }{3  }, \notag  \\
DE_6(\tau)=&\frac{E_2(\tau) E_6(\tau)-E_4(\tau)^2  }{2  },  \,\,D=\frac{1}{2\pi i}  \frac{d}{d  \tau}.  \label{eqn:Ramanujan-ODE-proof}
\end{align}

\subsection{On Klein's $J$-function}

\begin{theorem}
(Dedekind \cite{Dedekind})
\label{thm:ODE-J-j}
{\it
For every $\tau\in\mathbb{H}^2$ we have 
\begin{equation*}
\left\{J, \tau \right\}
+
\frac{36 J^2-41 J+32}{72  J^2 (J-1)^2} (J^{\prime})^2
=0.
\end{equation*}
}
\end{theorem}

\begin{proof}
We first set 
\begin{equation*}
P=E_2(\tau), Q=E_4(\tau), R=E_6(\tau) \,\,\mathrm{and} \,\,D=\frac{1}{2 \pi i} \frac{d}{d\tau},
\end{equation*}
which imply that 
\begin{equation*}
j=12^3\frac{Q^3}{Q^3-R^2}  \,\,\mathrm{and} \,\, \frac{Dj}{j}=-\frac{R}{Q}.
\end{equation*}
Direct computation yields 
\begin{equation*}
\frac{12^3}{j}=\frac{Q^3-R^2}{Q^3}=1-\frac{R^2}{Q^3}
=
1-\left(
\frac{Dj}{j}
\right)^2\frac{1}{Q},
\end{equation*}
which implies that 
\begin{equation}
\label{eqn:expression-Q-R}
Q=\frac{ (Dj)^2 }{j}\frac{1}{j-12^3} \,\,\mathrm{and} \,\, R=-\frac{ (Dj)^3 }{j^2}\frac{1}{j-12^3}. 
\end{equation}
\par
Taking the logarithmic differentiation of $Dj/j$ yields 
\begin{equation*}
\frac{D^2 j}{Dj}-\frac{D j}{j}=
\frac{1}{6} \left(P-\frac{3 Q^2}{R}+\frac{2 R}{Q}\right).
\end{equation*}
The differentiation of both sides implies that 
\begin{equation*}
\frac{D^3 j}{D j}
-
\left(
\frac{
D^2 j
}
{
D j
}
\right)^2
-
\frac{D^2 j}{j}
+
\left(
\frac{Dj}{j}
\right)^2
=
\frac{P^2}{72}-\frac{P Q^2}{12 R}+\frac{P R}{18 Q}-\frac{Q^4}{4 R^2}+\frac{R^2}{9 Q^2}+\frac{11 Q}{72}. 
\end{equation*}
Therefore it follows that 
\begin{align*}
&
\left[
\frac{D^3 j}{D j}
-
\left(
\frac{
D^2 j
}
{
D j
}
\right)^2
-
\frac{D^2 j}{j}
+
\left(
\frac{Dj}{j}
\right)^2
\right]
-
\frac12
\left(
\frac{D^2 j}{Dj}-\frac{D j}{j}
\right)^2  \\
=&
\frac{D^3 j}{D j}
-
\frac32
\left(
\frac{
D^2 j
}
{
D j
}
\right)^2
+
\frac12
\left(
\frac{Dj}{j}
\right)^2  
=
-\frac{3 Q^4}{8 R^2}+\frac{R^2}{18 Q^2}+\frac{23 Q}{72}.
\end{align*} 
Eq. (\ref{eqn:expression-Q-R}) yields 
\begin{align*}
\frac{D^3 j}{D j}
-
\frac32
\left(
\frac{
D^2 j
}
{
D j
}
\right)^2=&
-\frac{3 Q^4}{8 R^2}-\frac{4 R^2}{9 Q^2}+\frac{23 Q}{72}  
=
-\frac{\left(j^2-1968 j+2654208\right) (Dj)^2}{2 (j-1728)^2 j^2},
\end{align*}
which implies that 
\begin{equation}
\label{eqn:ODE-j}
\left\{
j,\tau
\right\}
+
\frac
{
j^2-2^4\cdot 3 \cdot 41 j+2^{15}\cdot 3^4
}
{
2 j^2(j-12^3)^2
}
(j^{\prime})^2
=0. 
\end{equation}
The theorem can be obtained by considering that $j=12^3 J.$ 
\end{proof}

\subsection{ On $E_4(\tau)$  }

\begin{proposition}
\label{prop:E_4-linear}
{\it 
The inverse function of J-function $\tau=\tau(J)$ yields $Q(J)=\sqrt[4]{ E_4(\tau(J))  }$ and 
$Q(J)$ satisfies the following differential equations:
\begin{equation}
\frac{d^2 Q}{d  J^2}
+
\left\{
\frac{\frac12}{J}+\frac{\frac12}{J-1}
\right\}
\frac{d Q}{dJ} 
-
\frac{5}{144J^2(J-1)}
Q
=0.
\end{equation}
}
\end{proposition}

\begin{proof}
We first set 
\begin{equation*}
P=E_2(\tau),  \,\,
Q=\sqrt[4]{ E_4(\tau) }  \,\,
\mathrm{and} \,\,
R=\sqrt[6]{E_6(\tau)},
\end{equation*}
which imply that 
\begin{align*}
DP=&
\frac{P^2-Q^4}{12}, \,\,
DQ=
\frac{P Q^4-R^6}{12 Q^3},  \,\,
DR=
\frac{P R^6-Q^8}{12 R^5} \,\,
\mathrm{and} \,\,
J=\frac{Q^{12}}{Q^{12}-R^{12}}. 
\end{align*}
\par
We next have 
\begin{equation*}
DJ=D
\left(
\frac{Q^{12}}{Q^{12}-R^{12}}
\right)
=
-\frac{Q^8 R^6}{Q^{12}-R^{12}},
\end{equation*}
\begin{equation*}
\frac{dQ}{d J}=
\frac{d Q}{d \tau} 
\frac{d \tau}{d J}
=
\frac{ D Q}{ D J}
=
-\frac{\left(Q^{12}-R^{12}\right) \left(P Q^4-R^6\right)}{12 Q^{11} R^6},  
\end{equation*}
and
\begin{equation*}
\frac{d^2 Q}{d J^2}=
\frac{d }{d \tau}
\left(
\frac{dQ}{d J}
\right)
\frac{d \tau}{ d J}
=
\frac{\left(Q^{12}-R^{12}\right)^2 \left(6 P Q^{16}+6 P Q^4 R^{12}-Q^{12} R^6-11 R^{18}\right)}{144 Q^{23} R^{18}},
\end{equation*}
which imply that 
\begin{equation*}
J(1-J)\frac{d^2 Q}{d J^2}
-
\frac12
J\frac{dQ}{d J}
+\frac12
(1-J)\frac{dQ}{d J}
=
-\frac{5 \left(Q^{12}-R^{12}\right)}{144 Q^{11}}
=
-\frac{5 Q}{144  J},
\end{equation*}
which proves the proposition. 
\end{proof}

\begin{theorem}
\label{thm:E_4-hypergeometric}
{\it
For every $\tau\in\mathbb{H}^2$ set 
$
\displaystyle 
x=
\frac{
1
}
{
J(\tau)
}.
$ 
Then the inverse function $\tau=\tau(x)$ yields $Q(x)=\sqrt[4]{ E_4(\tau(x))  }$ and 
$Q(x)$ satisfies the hypergeometric differential equations:
\begin{equation}
x(1-x)
\frac{d^2 Q}{dx^2}
+ \left(1-\frac32x \right) 
\frac{dQ}{dx} 
-
\frac{5}{144}
Q
=0.
\end{equation}
}
\end{theorem}

\begin{proof}
The theorem can be proved by 
changing 
$
\displaystyle
J\to \frac{1}{J}
$ 
in Proposition \ref{prop:E_4-linear}.
\end{proof}

\begin{theorem}
\label{thm:E_4-hypergeometric-function}
{\it
For every $\tau\in\mathbb{H}^2$ we have 
\begin{equation*}
\sqrt[4]{ E_4(\tau) } 
=
{ }_2 F_1 
\left(
\frac{1}{12},
\frac{5}{12}  ;
1; 
\frac{ 1}{ J(\tau) }
\right). 
\end{equation*}
}
\end{theorem}

\begin{proof}
Theorem \ref{thm:E_4-hypergeometric} implies that
\begin{equation*}
\sqrt[4]{ E_4(\tau) } =
A\cdot
{ }_2 F_1
\left(
\frac{1}{12}, 
\frac{5}{12} ;
1; 
x
\right)
+
B
\cdot
{ }_2 F_1
\left(
\frac{1}{12}, 
\frac{5}{12}; 
\frac12; 
1-x
\right),   \,\,\,
x=
\frac{ 1}{ J }, 
\end{equation*}
where $A$ and $B$ are constants. 
\par
Taking the limit $\tau \to +i \infty$ along the imaginary axis 
yields $x \to +0.$ 
For the proof see Apostol \cite[p. 41]{Apostol}. 
Since 
\begin{equation*}
 \lim_{x \to +0} Q(x)=1,
\end{equation*} 
Lemma \ref{lem:hypergeometric-limit-1-divergence} shows that $A=1$ and $B=0,$ 
which proves the theorem. 
\end{proof}

\subsection{ On $E_6(\tau)$  }
In this subsection we use the same notation as Proposition \ref{prop:E_4-linear}.

\begin{proposition}
\label{prop:E_6-linear}
{\it 
The inverse function of J-function $\tau=\tau(J)$ yields $R(J)=\sqrt[6]{ E_6(\tau(J))  }$ and 
$R(J)$ satisfies the following differential equations:
\begin{equation}
\frac{d^2 R}{d  J^2}
+
\left\{
\frac{\frac23}{J}+\frac{\frac13}{J-1}
\right\}
\frac{d R}{dJ} 
+
\frac{7}{144J(J-1)^2}
Q
=0.
\end{equation}
}
\end{proposition}

\begin{proof}
We first have 
\begin{equation*}
\frac{dR}{d J}=
\frac{d R}{d \tau} 
\frac{d \tau}{d J}
=
\frac{ D R}{ D J}
=
\frac{\left(Q^{12}-R^{12}\right) \left(Q^8-P R^6\right)}{12 Q^8 R^{11}},  
\end{equation*}
and
\begin{equation*}
\frac{d^2 R}{d J^2}=
\frac{d }{d \tau}
\left(
\frac{dR}{d J}
\right)
\frac{d \tau}{ d J}
=
-\frac{\left(Q^{12}-R^{12}\right)^2 \left(-4 P Q^{12} R^6-8 P R^{18}+11 Q^{20}+Q^8 R^{12}\right)}{144 Q^{20} R^{23}}, 
\end{equation*}
which imply that 
\begin{equation*}
J(J-1)\frac{d^2 R}{d J^2}
+
\frac13
J\frac{dR}{d J}
+\frac23
(J-1)\frac{dR}{d J}
=
\frac{7 \left(R^{12}-Q^{12}\right)}{144 R^{11}}
=
-\frac{7  R}{144 (J-1) },
\end{equation*}
which proves the proposition. 
\end{proof}

\begin{theorem}
\label{thm:E_6-hypergeometric}
{\it
For every $\tau\in\mathbb{H}^2$ set 
$
\displaystyle 
y=
\frac{
1
}
{
1-J(\tau)
}.
$ 
Then the inverse function $\tau=\tau(y)$ yields $R(y)=\sqrt[6]{ E_6(\tau(y))  }$ and 
$R(y)$ satisfies the hypergeometric differential equations:
\begin{equation}
y(1-y)
\frac{d^2 R}{dy^2}
+ \left(1-\frac53 y \right) 
\frac{dR}{dy} 
-
\frac{7}{144}
R
=0.
\end{equation}
}
\end{theorem}

\begin{proof}
The theorem can be proved by 
changing 
$
\displaystyle
J\to \frac{1}{1-J}
$ 
in Proposition \ref{prop:E_6-linear}.
\end{proof}

\begin{theorem}
\label{thm:E_6-hypergeometric-function}
{\it
For every $\tau\in\mathbb{H}^2$ we have 
\begin{equation*}
\sqrt[6]{ E_6(\tau) } 
=
{ }_2 F_1 
\left(
\frac{1}{12},  
\frac{7}{12}  ;
1; 
\frac{ 1}{ 1-J(\tau) }
\right). 
\end{equation*}
}
\end{theorem}

\begin{proof}
Theorem \ref{thm:E_6-hypergeometric} implies that
\begin{equation*}
\sqrt[6]{ E_6(\tau) } =
A\cdot
{ }_2 F_1
\left(
\frac{1}{12}, 
\frac{7}{12};
1; 
y
\right)
+
B
\cdot
(1-y)^{ -\frac{1}{12} }
{ }_2 F_1
\left(
\frac{1}{12},
\frac{5}{12} ;
\frac12; 
\frac{1}{1-y}
\right),   \,\,\,
y=
\frac{ 1}{ 1-J }, 
\end{equation*}
where $A$ and $B$ are constants. 
\par
Taking the limit $\tau \to +i \infty$ along the imaginary axis 
yields $y \to -0.$ 
Since 
\begin{equation*}
 \lim_{y \to -0} R(y)=1,
\end{equation*} 
Lemma \ref{lem:hypergeometric-limit-1-divergence} shows that $A=1$ and $B=0,$ 
which proves the theorem. 
\end{proof}

\section{The Hypergeometric equations and the theta constants}
\label{sec:hypergeometric-theta}
In \cite{Matsuda3}, we proved the following theorem:
\begin{theorem}(Halphen)
\label{thm:Halphen-rewrite}
{\it
For every $\tau\in\mathbb{H}^2$, set
\begin{align*}
\left(
P(\tau),
Q(\tau),
R(\tau)
\right)
=&
\left(
\theta^4
\left[
\begin{array}{c}
0 \\
0
\end{array}
\right](0,\tau), \,
\theta^4
\left[
\begin{array}{c}
1 \\
0
\end{array}
\right](0,\tau), \,
E_2(\tau)
\right), \\
&
\left(
-
\theta^4
\left[
\begin{array}{c}
1 \\
0
\end{array}
\right](0,\tau), \,
\theta^4
\left[
\begin{array}{c}
0 \\
1
\end{array}
\right](0,\tau), \,
E_2(\tau)
\right), \\
&
\left(
-
\theta^4
\left[
\begin{array}{c}
0 \\
1
\end{array}
\right](0,\tau), \,
-
\theta^4
\left[
\begin{array}{c}
0 \\
0
\end{array}
\right](0,\tau), \,
E_2(\tau)
\right).
\end{align*}
Then, we have
\begin{equation*}
DP=
\frac{-P^2 + 2 P Q + P R}{6},  \,\,
DQ=
\frac
{ - Q^2+2 P Q + Q R}{6}, \,\,
DR=
\frac{-P^2 + P Q - Q^2 + R^2}{12}, \,\,\,\left(D=\frac{1}{2\pi i} \frac{d}{d\tau}\right).
\end{equation*}
}
\end{theorem}

\subsection{On $\lambda$-function  }

\begin{theorem}
\label{thm:ODE-lambda}
{\it
For every $\tau\in\mathbb{H}^2$ we have 
\begin{equation*}
\left\{  \lambda, \tau  \right\}
+
\frac
{
\lambda^2-\lambda+1
}
{
2\lambda^2(\lambda-1)^2
}
\left(
\lambda^{\prime}
\right)^2=0. 
\end{equation*}
}
\end{theorem}

\begin{proof}
We recall the ODE (\ref{eqn:ODE-j}) and the well-known property of the Schwartzian derivative,  
\begin{equation*}
\left\{  j, \tau  \right\}=\left\{  j, \lambda  \right\} (\lambda^{\prime})^2 +\left\{  \lambda, \tau  \right\} \,\,
\mathrm{and} \,\,
j=2^8
\frac
{
(\lambda^2-\lambda+1)^3
}
{
\lambda^2(\lambda-1)^2
}. 
\end{equation*}
Direct calculation yields 
\begin{align*}
\frac{d j}{d\lambda}=&
\frac{256 \left(\lambda ^2-\lambda +1\right)^2 \left(2 \lambda ^3-3 \lambda ^2-3 \lambda +2\right)}{(\lambda -1)^3 \lambda ^3}, \\
\frac{d^2 j}{d\lambda^2}=&
\frac{512 \left(\lambda ^8-4 \lambda ^7+6 \lambda ^6-4 \lambda ^5+10 \lambda ^4-18 \lambda ^3+22 \lambda ^2-13 \lambda +3\right)}{(\lambda -1)^4 \lambda ^4}, \\
\frac{d^3 j}{d\lambda^3}=&
-\frac{1536 \left(12 \lambda ^5-30 \lambda ^4+50 \lambda ^3-45 \lambda ^2+21 \lambda -4\right)}{(\lambda -1)^5 \lambda ^5},
\end{align*} 
which implies that 
\begin{align*}
\left\{  j, \lambda  \right\}
=&
-\frac{6 \left(\lambda ^{12}-6 \lambda ^{11}+13 \lambda ^{10}-10 \lambda ^9+40 \lambda ^8-166 \lambda ^7+257 \lambda ^6-166 \lambda ^5+40 \lambda ^4-10 \lambda ^3+13 \lambda ^2-6 \lambda +1\right)}{(1-2 \lambda )^2 (\lambda -1)^2 \lambda ^2 \left(-\lambda
   ^2+\lambda +2\right)^2 \left(\lambda ^2-\lambda +1\right)^2}. 
\end{align*}
\par
Moreover we have 
\begin{align*}
&\frac
{
j^2-2^4\cdot 3 \cdot 41 j+2^{15}\cdot 3^4
}
{
2 j^2(j-12^3)^2
}
(j^{\prime})^2=
\frac
{
j^2-2^4\cdot 3 \cdot 41 j+2^{15}\cdot 3^4
}
{
2 j^2(j-12^3)^2
}
\left(
\frac{d j}{d\lambda}
\frac{d \lambda}{d \tau}
\right)^2  \\
=&
\frac{\left(2 \lambda ^3-3 \lambda ^2-3 \lambda +2\right)^2 }{2 (\lambda -2)^4 (\lambda -1)^2 \lambda ^2 (\lambda +1)^4 (2 \lambda -1)^4 \left(\lambda ^2-\lambda +1\right)^2} \times \\
&
\times
\Big(
16 \lambda ^{12}-96 \lambda ^{11}+213 \lambda ^{10}-185 \lambda ^9+489 \lambda ^8-1902 \lambda ^7+2946 \lambda ^6 \\
&\hspace{45mm}
-1902 \lambda ^5+489 \lambda ^4-185 \lambda ^3+213 \lambda ^2-96 \lambda
   +16
\Big)(\lambda^{\prime})^2,
\end{align*}
which implies that 
\begin{align*}
\left\{  j, \lambda  \right\} (\lambda^{\prime})^2+
\frac
{
j^2-2^4\cdot 3 \cdot 41 j+2^{15}\cdot 3^4
}
{
2 j^2(j-12^3)^2
}
(j^{\prime})^2
=
\frac
{
\lambda^2-\lambda+1
}
{
2\lambda^2(\lambda-1)^2
}
\left(
\lambda^{\prime}
\right)^2.
\end{align*}
Therefore the theorem follows from Theorem \ref{thm:ODE-J-j}.
\end{proof}

\subsection{On the theta constants (1)  }
In this subsection we set 
\begin{equation*}
P=
\theta^2
\left[
\begin{array}{c}
0 \\
0
\end{array}
\right](0,\tau),  \,\,
Q=
\theta^2
\left[
\begin{array}{c}
1 \\
0
\end{array}
\right](0,\tau), \,\, 
\mathrm{and}  \,\,
R=E_2(\tau).
\end{equation*}

\begin{theorem}
\label{thm:theta-(0,0)-lambda-hypergeometric}
{\it
For every $\tau\in\mathbb{H}^2$ set 
$
\displaystyle 
\lambda=
\theta^4
\left[
\begin{array}{c}
1 \\
0
\end{array}
\right](0,\tau)
/
\theta^4
\left[
\begin{array}{c}
0 \\
0
\end{array}
\right](0,\tau). 
$ 
Then the inverse function $\tau=\tau(\lambda)$ yields $P=P(\lambda)=P(\tau(\lambda))$ and 
$P(\lambda)$ satisfies the hypergeometric differential equations:
\begin{equation}
\lambda(1-\lambda)
\frac{d^2 P}{d\lambda^2}
+(1-2\lambda) 
\frac{dP}{d\lambda} 
-
\frac14
P
=0.
\end{equation}
}
\end{theorem}

\begin{proof}
Halphen's system (\ref{thm:Halphen-rewrite}) implies that 
\begin{equation*}
DP=\frac{-P^3+2 P Q^2+P R}{12}, \,\,
DQ=\frac{2 P^2 Q-Q^3+Q}{12},  \,\,
DR=\frac{-P^4+P^2 Q^2-Q^4+R^2}{12}. 
\end{equation*}
Then we have 
\begin{equation*}
D \lambda=D\left(\frac{Q^2}{P^2} \right)=
-\frac{Q^2 (Q-P) (P+Q)}{2 P^2}, 
\end{equation*}
\begin{equation*}
\frac{dP}{d  \lambda}=
\frac{d P}{d \tau} 
\frac{d\tau}{d \lambda }
=\frac{D P}{D \lambda}
=
-\frac{P^2 \left(-P^3+2 P Q^2+P R\right)}{6 Q^2 (Q-P) (P+Q)},  
\end{equation*}
and 
\begin{equation*}
\frac{d^2 P}{d\lambda^2}
=
\frac{d}{d\tau}  \left( \frac{dP}{d  \lambda} \right) \frac{d\tau}{d\lambda} 
=
\frac{P^5 \left(2 P^4+5 Q^4+4Q^2 R-5P^2 Q^2-2 P^2 R \right)}{12 Q^4 \left(P^2-Q^2\right)^2}, 
\end{equation*}
which imply that 
\begin{align*}
&
\lambda(1-\lambda)
\frac{ d^2 P}{d \lambda^2}
-
\lambda
\frac{d P}{d \lambda}
+
(1-\lambda)
\frac{d P}{d \lambda}
=
\frac14 P,
\end{align*}
which proves the theorem. 
\end{proof}

\begin{theorem}
\label{thm:theta-(0,0)-lambda-hypergeometric-function}
{\it
For every $\tau\in\mathbb{H}^2$ we have 
\begin{equation*}
\theta^2
\left[
\begin{array}{c}
0 \\
0
\end{array}
\right](0,\tau)
=
{ }_2 F_1 
\left(
\frac12, 
\frac12; 
1; 
\lambda(\tau)
\right). 
\end{equation*}
}
\end{theorem}

\begin{proof}
Theorem \ref{thm:theta-(0,0)-lambda-hypergeometric} implies that
\begin{equation*}
P(\lambda)=
A\cdot
{ }_2 F_1
\left(
\frac12, 
\frac12; 
1; 
\lambda
\right)
+
B
\cdot
{ }_2 F_1
\left(
\frac12, 
\frac12; 
1; 
1-\lambda
\right),   
\end{equation*}
where $A$ and $B$ are constants. 
\par
Taking the limit $\tau \to +i \infty$ along the imaginary axis 
yields $\lambda \to +0.$ 
Since 
\begin{equation*}
 \lim_{ \lambda \to +0} P(\lambda)=1,
\end{equation*} 
Lemma \ref{lem:hypergeometric-limit-1-divergence} shows that $A=1$ and $B=0,$ 
which proves the theorem. 
\end{proof}

\subsubsection*{Remark}
The Jacobi's inversion formula of modular forms of level four can be obtained by changing 
$\tau \to 2 \tau,$ or $\tau \to \tau/2$ in Theorem \ref{thm:theta-(0,0)-lambda-hypergeometric-function}. 
Recall that 
\begin{equation*}
j\left(  \frac{\tau}{2} \right)
=
2^4
\frac
{
(\lambda^2+14\lambda+1)^3
}
{
\lambda (\lambda-1)^4
}, \,\,
\mathrm{or} \,\,
j(2\tau)
=
2^{16}
\frac
{
\left(1-\lambda+\frac{1}{16}\lambda^2\right)^3
}
{
\lambda^4 (1-\lambda)
}.
\end{equation*}

\subsection{On the theta constants (2)  }

In this subsection we use the same notations as Theorem \ref{thm:theta-(0,0)-lambda-hypergeometric}.

\begin{proposition}
\label{prop:theta-(1,0)-lambda}
{\it
For every $\tau\in\mathbb{H}^2$ set 
$
\displaystyle 
\lambda=
\theta^4
\left[
\begin{array}{c}
1 \\
0
\end{array}
\right](0,\tau)
/
\theta^4
\left[
\begin{array}{c}
0 \\
0
\end{array}
\right](0,\tau). 
$ 
Then the inverse function $\tau=\tau(\lambda)$ yields $Q=Q(\lambda)=Q(\tau(\lambda))$ and 
$Q(\lambda)$ satisfies the hypergeometric differential equations:
\begin{equation}
\frac{d^2 Q}{d \lambda^2}
+
\frac{1}{\lambda-1}
\frac{dQ}{d\lambda} 
-
\frac{1}{4 \lambda^2(\lambda-1)}
Q
=0.
\end{equation}
}
\end{proposition}

\begin{proof}
We first have 
\begin{equation*}
\frac{d Q}{d  \lambda}=
\frac{d Q}{d \tau} 
\frac{d\tau}{d \lambda }
=\frac{D Q}{D \lambda}
=
\frac{P^2 \left(2 P^2-Q^2+R\right)}{6 Q \left(P^2-Q^2\right)}, 
\end{equation*}
and 
\begin{equation*}
\frac{d^2 Q}{d\lambda^2}
=
\frac{d}{d\tau}  \left( \frac{dQ}{d  \lambda} \right) \frac{d\tau}{d\lambda} 
=
-\frac{P^4 \left(3 P^4-7 P^2 Q^2+2 Q^4-2 Q^2 R\right)}{12 Q^3 (P-Q)^2 (P+Q)^2},  
\end{equation*}
which imply that 
\begin{align*}
&
\lambda(\lambda-1)
\frac{ d^2 P}{d \lambda^2  }
+
\lambda
\frac{d Q}{d \lambda}
=
\frac{P^2}{4 Q}=
\frac{Q}{4 \lambda},
\end{align*}
which proves the proposition. 
\end{proof}

\begin{theorem}
\label{thm:theta-(1,0)-lambda-hypergeometric}
{\it
For every $\tau\in\mathbb{H}^2$ set 
$
\displaystyle 
x=\frac{1}{\lambda}
$ 
Then the inverse function $\tau=\tau(x)$ yields $Q=Q(x)=P(\tau(x))$ and 
$Q(x)$ satisfies the hypergeometric differential equations:
\begin{equation}
\label{eqn:ODE-(1,0)}
x(1-x)
\frac{d^2Q}{dx^2}
+(1-2x) 
\frac{dQ}{dx} 
-
\frac14
Q
=0.
\end{equation}
}
\end{theorem}

\begin{proof}
The theorem can be proved by changing 
$\displaystyle  \lambda \to \frac{1}{\lambda}$ in Proposition \ref{prop:theta-(1,0)-lambda}. 
\end{proof}

\begin{theorem}
\label{thm:theta-(1,0)-lambda-hypergeometric-function}
{\it
For every $\tau\in\mathbb{H}^2$ we have 
\begin{equation*}
\theta^2
\left[
\begin{array}{c}
1 \\
0
\end{array}
\right](0,\tau)
=
\left(
\frac{1}{\lambda}
\right)^{-\frac12}
{ }_2 F_1 
\left(
\frac12, 
\frac12; 
1; 
\frac{1}{1/\lambda(\tau)}
\right). 
\end{equation*}
}
\end{theorem}

\begin{proof}
The formula follows from Theorem \ref{thm:theta-(0,0)-lambda-hypergeometric-function}. 
Note that $
\displaystyle
x^{-\frac12}  { }_2 F_1 
\left(
\frac12, 
\frac12; 
1; 
\frac{1}{x}
\right) $ is a solution of the ODE (\ref{eqn:ODE-(1,0)}). 
\end{proof}

\subsection{On the theta constants (3)  }
In this subsection we set 
\begin{equation*}
P=
\theta^2
\left[
\begin{array}{c}
0 \\
0
\end{array}
\right](0,\tau),  \,\,
Q=
\theta^2
\left[
\begin{array}{c}
0 \\
1
\end{array}
\right](0,\tau), \,\, 
\mathrm{and}  \,\,
R=E_2(\tau).
\end{equation*}

\begin{theorem}
\label{thm:theta-(0,1)-lambda-hypergeometric}
{\it
For every $\tau\in\mathbb{H}^2$ set 
$
\displaystyle 
y=
\frac{1}{1-\lambda}. 
$ 
Then the inverse function $\tau=\tau( y)$ yields $Q=Q(y)=Q(\tau( y))$ and 
$Q(y)$ satisfies the hypergeometric differential equations:
\begin{equation}
y(1-y)
\frac{d^2 Q}{d  y^2}
+(1-2 y) 
\frac{dQ}{d y } 
-
\frac14
Q
=0.
\end{equation}
}
\end{theorem}

\begin{proof}
Halphen's system (\ref{thm:Halphen-rewrite}) implies that 
\begin{equation*}
DP=\frac{P^3-2 P Q^2+P R}{12}, \,\,
DQ=\frac{-2 P^2 Q+Q^3+Q R}{12}, \,\,
DR=\frac{-P^4+P^2 Q^2-Q^4+R^2}{12}. 
\end{equation*}
Then we have 
\begin{equation*}
D y=D\left(\frac{P^2}{Q^2} \right)=
\frac{P^2 (P-Q) (P+Q)}{2 Q^2}, 
\end{equation*}
\begin{equation*}
\frac{dQ}{d  y}=
\frac{d Q}{d \tau} 
\frac{d\tau}{d y }
=\frac{D Q}{D y }
=
\frac{Q^2 \left(-2 P^2 Q+Q^3+Q R\right)}{6 P^2 (P-Q) (P+Q)},   
\end{equation*}
and 
\begin{equation*}
\frac{d^2 Q}{d  y^2}
=
\frac{d}{d\tau}  \left( \frac{dQ}{d  y } \right) \frac{d\tau}{d  y  } 
=
\frac{Q^5 \left(5 P^4-5 P^2 Q^2-4 P^2 R+2 Q^4+2 Q^2 R\right)}{12 P^4 (P-Q)^2 (P+Q)^2},  
\end{equation*}
which imply that 
\begin{align*}
&
y(1-y)
\frac{ d^2 Q}{d y^2}
-
y
\frac{d Q}{d \tilde{\lambda  }}
+
(1-y)
\frac{d Q}{d y }
=
\frac14 Q,
\end{align*}
which proves the theorem. 
\end{proof}

\begin{theorem}
\label{thm:theta-(0,1)-lambda-hypergeometric-function}
{\it
For every $\tau\in\mathbb{H}^2$ we have 
\begin{equation*}
\theta^2
\left[
\begin{array}{c}
0 \\
1
\end{array}
\right](0,\tau)
=
{ }_2 F_1 
\left(
\frac12, 
\frac12; 
1; 
\frac{\lambda(\tau)}{\lambda(\tau)-1} 
\right). 
\end{equation*}
}
\end{theorem}

\begin{proof}
Theorem \ref{thm:theta-(0,1)-lambda-hypergeometric} implies that
\begin{equation*}
Q( y)=
A\cdot
{ }_2 F_1
\left(
\frac12, 
\frac12; 
1; 
1-y
\right)
+
B
\cdot
(-y)^{-\frac12}
{ }_2 F_1
\left(
\frac12, 
\frac12; 
1; 
\frac{1}{y}
\right),   
\end{equation*}
where $A$ and $B$ are constants. 
\par
Taking the limit $\tau \to +i \infty$ along the imaginary axis 
yields $y \to 1+0.$ 
Since 
\begin{equation*}
 \lim_{ y \to 1+0} Q(y)=1,
\end{equation*} 
Lemma \ref{lem:hypergeometric-limit-1-divergence} shows that $A=1$ and $B=0,$ 
which proves the theorem. 
\end{proof}

\subsubsection*{Remark}
The formula of Theorem \ref{thm:theta-(0,1)-lambda-hypergeometric-function} can also be obtained by 
changing $\tau \to \tau+1$ in Theorem \ref{thm:theta-(0,0)-lambda-hypergeometric-function}.

\section{The Hypergeometric equations and the cubic theta functions}
\label{sec:hypergeometric-cubic}

In \cite{Matsuda1}, we derived the systems of ODEs, 
\begin{align}
Da(\tau)=&\frac{3a^3(\tau)+a(\tau)E_2(\tau)-4b^3(\tau)}{12},   \,\,\,
D E_2(\tau)=\frac{-9a^4(\tau)+8a(\tau)b^3(\tau)+(E_2(\tau))^2}{12},  \notag \\
Db^3(\tau)=&\frac{ -a^2(\tau)b^3(\tau)+E_2(\tau) b^3(\tau) }{4   },  \label{eqn:ODE-a-b^3}
\end{align}
and 
\begin{align}
Da(\tau)=&
\frac
{
-3 a^3(\tau)+3 a(\tau) E_2(3\tau)+4 c^3(\tau)
}
{
12
}, \,\,
DE_2(3\tau)
=
\frac
{
-9 a^4(\tau)+8 a(\tau) c^3(\tau)+9 E_2(3\tau)^2
}
{
36
},   \notag   \\
Dc^3(\tau)=&
\frac
{
a^2(\tau) c^3(\tau)+3 E_2(3\tau) c^3(\tau)
}
{
4
},   \,\,\,D=\frac{1}{2 \pi i} \frac{d}{d\tau}.   \label{eqn:ODE-a-c^3}
\end{align}
\par
From Cooper \cite[p. 187, p. 184]{Cooper}, 
we recall the following formulas, 
\begin{align}
a(q)
=&\sum_{m,n\in\mathbb{Z}} q^{m^2+mn+n^2}
=1+6\sum_{n=1}^{\infty} (d_{1,3}(n)-d_{2,3}(n)) q^n,      \label{eqn:formula-a(q)} \\
b(q)
=&\sum_{m,n\in\mathbb{Z}} \omega^{n-m} q^{m^2+mn+n^2}=\prod_{n=1}^{\infty} \frac{ (1-q^n)^3  }{  (1-q^{3n})  }, \label{eqn:formula-b(q)}   \\
c(q)=&\sum_{m,n\in\mathbb{Z}} q^{(n+\frac13)^2+(n+\frac13)(m+\frac13)+(m+\frac13)^2}
=
3q^{\frac13}
\prod_{n=1}^{\infty} \frac{ (1-q^{3n})^3  }{  (1-q^{n})  }.  \label{eqn:formula-c(q)}
\end{align}

\subsection{On the modular function of level three}

\begin{theorem}
\label{thm:ODE-x-(1,1/3)}
{\it
For every $\tau\in\mathbb{H}^2$ set 
$
\displaystyle 
x(\tau)=
\frac{
c^3(\tau)
}
{
a^3(\tau)
}.
$
Then we have 
\begin{equation*}
\{x,\tau\}
+
\frac{8 x^2-8 x+9}{18  x^2  (x-1)^2}
\left(x^{\prime} \right)^2=0. 
\end{equation*}
}
\end{theorem}

\begin{proof}
We recall the ODE (\ref{eqn:ODE-j}) and the well-known property of the Schwartzian derivative, 
\begin{equation*}
\{j, \tau \}
=\{j, x \}
\left(
\frac{d x}{d\tau}
\right)^2+\{x, \tau\} \,\,
\mathrm{and} \,\,
j=
27
\frac
{
\left( 1+8 x \right)^3
}
{
x (1-x)^3
}.
\end{equation*}
The formula of $j(\tau)$ is written in Cooper \cite[p. 272]{Cooper}. 
\par
Direct computation yields 
\begin{equation*}
\frac{dj}{d x}
=
\frac{27 (8 x+1)^2 \left(8 x^2+20 x-1\right)}{(x-1)^4 x^2},  
\end{equation*}
\begin{equation*}
\frac{d^2j}{d x^2}
=
-\frac{54 (8 x+1) \left(64 x^4+320 x^3+114 x^2-13 x+1\right)}{(x-1)^5 x^3}, 
\end{equation*}
and
\begin{align*}
\frac{d^3j}{d x^3}
=&
\frac{162 \left(512 x^6+3840 x^5+2928 x^4+20 x^3-15 x^2+6 x-1\right)}{ x^4 (x-1)^6 }, 
\end{align*}
which imply that 
\begin{equation*}
\{j, x\}
=
-\frac{324 \left(320 x^4+304 x^3+138 x^2-38 x+5\right)}{(x-1)^2 (8 x+1)^2 \left(8 x^2+20 x-1\right)^2}. 
\end{equation*} 
Moreover we have 
\begin{align*}
&
\frac
{
j^2-2^4\cdot 3 \cdot 41 j+2^{15}\cdot 3^4
}
{
2 j^2(j-12^3)^2
}
(j^{\prime})^2
=
\frac
{
j^2-2^4\cdot 3 \cdot 41 j+2^{15}\cdot 3^4
}
{
2 j^2(j-12^3)^2
}
\left(
\frac{dj}{d x}
\frac{d x}{d\tau}
\right)^2  \\
=&
\frac{32768 x^8+139264 x^7+1969152 x^6+1759616 x^5+1039888 x^4-185136 x^3+27632 x^2-224 x+9}{18 (x-1)^2 x^2 (8 x+1)^2 \left(8 x^2+20 x-1\right)^2}
(x^{\prime})^2. 
\end{align*}
Therefore it follows that 
\begin{equation*}
\{x,\tau\}
+
\frac{8 x^2-8 x+9}{18  x^2 (x-1)^2}
\left(x^{\prime} \right)^2=0. 
\end{equation*}
\end{proof}

\subsection{On $a(\tau)$}

\begin{theorem}
\label{thm:a-x-hypergeometric-E_2(q)}
{\it
For every $\tau\in\mathbb{H}^2$ set 
$
\displaystyle 
x(\tau)=
\frac{
c^3(\tau)
}
{
a^3(\tau)
}.
$ 
Then the inverse function $\tau=\tau(x)$ yields $a=a(x)=a(\tau(x))$ and 
$a(x)$ satisfies the hypergeometric differential equations:
\begin{equation}
x(1-x)
\frac{d^2 a}{dx^2}
+(1-2x) 
\frac{da}{dx} 
-
\frac29 
a
=0.
\end{equation}
}
\end{theorem}

\begin{proof}
We set 
\begin{equation*}
P=a(\tau), 
Q=E_2(\tau), 
R=b(\tau) 
\,\,
\mathrm{and} \,\,
D=\frac{1}{2\pi i}
\frac{d}{d \tau}. 
\end{equation*}
Eq. (\ref{eqn:ODE-a-b^3}) implies that 
\begin{equation*}
DP=\frac{ 3 P^3+ P Q-4 R^3 }{12}, \,\,
DQ=\frac{-9 P^4+8 P R^3+Q^2}{12}, \,\,
DR=\frac{ -P^2 R+ Q R  }{12}. 
\end{equation*}
Then we have 
\begin{equation*}
D x=D\left(1-\frac{R^3}{P^3} \right)=
\frac{R^3 \left(P^3-R^3\right)}{P^4},  
\end{equation*}
\begin{equation*}
\frac{d a}{dx}=
\frac{d a}{d\tau} 
\frac{d\tau}{d x}
=\frac{D P}{D x}
=
\frac{P^4 \left(3 P^3+P Q-4 R^3\right)}{12 R^3 \left(P^3-R^3\right)}, 
\end{equation*}
and 
\begin{equation*}
\frac{d^2 a}{dx^2}
=
\frac{d}{d\tau}  \left( \frac{d a}{dx}  \right) \frac{d\tau}{dx} 
=
\frac{P^3 \left(9 P^6+3 P^4 Q-22 P^3 R^3-6 P Q R^3+16 R^6\right)}{36 R^3 \left(P^3-R^3\right)},
\end{equation*}
which imply that 
\begin{align*}
&
x(1-x)
\frac{ d a}{d x}
-
x
\frac{d a}{d x}
+
(1-x)
\frac{d a}{d x}
=
\frac29 P,
\end{align*}
which proves the theorem. 
\end{proof}

\begin{theorem}
\label{thm:a(q)-hypergeometric}
{\it
For every $\tau\in\mathbb{H}^2$ we have 
\begin{equation*}
a(\tau)
=
{ }_2 F_1 
\left(
\frac13, 
\frac23; 
1; 
\frac{ c^3(\tau)}{ a^3(\tau) }
\right). 
\end{equation*}
}
\end{theorem}

\begin{proof}
Theorem \ref{thm:a-x-hypergeometric-E_2(q)} implies that
\begin{equation*}
a(x)=
A\cdot
{ }_2 F_1
\left(
\frac13, 
\frac23; 
1; 
x
\right)
+
B
\cdot
{ }_2 F_1
\left(
\frac13, 
\frac23; 
1; 
1-x
\right),   \,\,\,
x=
\frac{ c^3 (\tau) }{ a^3(\tau)  }, 
\end{equation*}
where $A$ and $B$ are constants. 
\par
From the formulas (\ref{eqn:formula-a(q)}), (\ref{eqn:formula-b(q)}) and (\ref{eqn:formula-c(q)}), 
taking the limit $\tau \to +i \infty$ along the imaginary axis 
yields $x \to +0.$ 
Since 
\begin{equation*}
 \lim_{x \to +0} a(x)=1,
\end{equation*} 
Lemma \ref{lem:hypergeometric-limit-1-divergence} shows that $A=1$ and $B=0,$ 
which proves the theorem. 
\end{proof}

\subsection{On $b(\tau)$}

\begin{proposition}
\label{prop:b-x-linear-E_2(q)}
{\it
For every $\tau\in\mathbb{H}^2$ set 
$
\displaystyle 
x(\tau)=
\frac{
c^3(\tau)
}
{
a^3(\tau)
}.
$ 
Then the inverse function $\tau=\tau(x)$ yields $b=b(x)=b(\tau(x))$ and 
$b(x)$ satisfies the following differential equations:
\begin{equation}
\frac{d^2 b}{dx^2}
+
\left\{
\frac{1}{x}+\frac{1}{3(x-1)}
\right\}
\frac{db}{dx} 
+
\frac{1}{9x(x-1)^2}
b
=0.
\end{equation}
}
\end{proposition}

\begin{proof}
We use the same notation as Theorem \ref{thm:a-x-hypergeometric-E_2(q)}. 
Then we have 
\begin{equation*}
\frac{ d b}{d x}=\frac{d b}{d\tau} \frac{d \tau}{d x} =\frac{ D R}{D x}
=
\frac{P^4 \left(Q R-P^2 R\right)}{12 R^3 \left(P^3-R^3\right)}, 
\end{equation*}
\begin{equation*}
\frac{ d^2 b}{d x^2}=
\frac{d}{d\tau} \left( \frac{d b}{d x} \right)
\frac{d x}{d \tau}
=
\frac{P^7 \left(-5 P^5+P^3 Q+8 P^2 R^3-4 Q R^3\right)}{36 R^5 \left(P^3-R^3\right)^2},
\end{equation*}
which imply that 
\begin{align*}
&
x(1-x)\frac{ d^2 b}{d x^2}
-
\frac13 x 
\frac{ d b}{d x}
+
(1-x)
\frac{ d b}{d x}
=
-\frac{P^3}{9 R^2}  \\
=&
-\frac{P^3}{9 R^3} R=
-\frac{a^3(\tau)}{9 b^3(\tau)} b(\tau) 
=
-\frac19 \frac{a^3(\tau)}{a^3(\tau)-c^3(\tau)} b(\tau)
=
-\frac19 \frac{1}{1-x} b.
\end{align*}
\end{proof}

\begin{theorem}
\label{thm:b-y-hypergeometric-E_2(q)}
{\it
For every $\tau\in\mathbb{H}^2$ set 
$
\displaystyle 
y(\tau)=
\frac{1}
{
1
-
\frac{c^3(\tau)}{a^3(\tau)}
}
=
\frac
{
a^3(\tau)
}
{
b^3(\tau)
}. 
$ 
Then the inverse function $\tau=\tau(y)$ yields $b=b(y)=b(\tau(y))$ and 
$b(y)$ satisfies the hypergeometric diferential equation:
\begin{equation}
y(1-y)
\frac{d^2 b}{dy^2}
+
\left(
\frac23-\frac53 y
\right)
\frac{db}{dy} 
-
\frac19
b
=0.
\end{equation}
}
\end{theorem}

\begin{proof}
The theorem can be proved by changing $\displaystyle x\to \frac{1}{1-x  }$ 
in Proposition \ref{prop:b-x-linear-E_2(q)}. 
\end{proof}

\begin{theorem}
\label{thm:b(q)-hypergeometric}
{\it
For every $\tau\in\mathbb{H}^2$ we have 
\begin{equation*}
b(\tau)
=
{ }_2 F_1
\left(
\frac13, 
\frac13; 
1; 
1-
\frac
{
a^3(\tau)
}
{
b^3(\tau)
}
\right).
\end{equation*}
}
\end{theorem}

\begin{proof}
Theorem \ref{thm:b-y-hypergeometric-E_2(q)} implies that
\begin{equation*}
b(y)=
A\cdot
(-y)^{-\frac13}
{ }_2 F_1
\left(
\frac13, 
\frac23; 
1; 
\frac{1}{y}
\right)
+
B
\cdot
{ }_2 F_1
\left(
\frac13, 
\frac13; 
1; 
1-y
\right),   \,\,\,
y=
\frac
{
a^3(\tau)
}
{
b^3(\tau)
}, 
\end{equation*}
where $A$ and $B$ are constants. 
\par
From the formulas (\ref{eqn:formula-a(q)}), (\ref{eqn:formula-b(q)}) and (\ref{eqn:formula-c(q)}), 
taking the limit $\tau \to +i \infty$ along the imaginary axis 
yields $y \to 1+0.$ 
Since 
\begin{equation*}
\lim_{ \tau \to +i \infty  } b(\tau)=1, 
\end{equation*}
Lemma \ref{lem:hypergeometric-limit-1-divergence} shows that $A=0$ and $B=1,$  
which proves the theorem. 
\end{proof}

\subsection{On $c(\tau)$}

Let us consider ODE satisfied by $c(\tau).$

\begin{proposition}
\label{prop:c-x-linear-E_2(q)}
{\it
For every $\tau\in\mathbb{H}^2$ set 
$
\displaystyle 
x(\tau)=
\frac{
c^3(\tau)
}
{
a^3(\tau)
}.
$ 
Then the inverse function $\tau=\tau(x)$ yields $c=c(x)=c(\tau(x))$ and 
$c(x)$ satisfies the following differential equations:
\begin{equation}
\frac{d^2 c}{dx^2}
+
\left\{
\frac{1}{3x}+\frac{1}{(x-1)}
\right\}
\frac{dc}{dx} 
-
\frac{1}{9x^2(x-1)}
c
=0.
\end{equation}
}
\end{proposition}

\begin{proof}
We set 
\begin{equation*}
P=a(\tau),  \,\,
Q=E_2(3\tau), \,\,
R=c(\tau) 
\,\,
\mathrm{and} \,\,
D=\frac{1}{2\pi i} 
\frac{d}{d \tau}. 
\end{equation*}
Eq. (\ref{eqn:ODE-a-c^3}) implies that 
\begin{equation*}
DP=\frac
{
-3 P^3+ 3 PQ+4 R^3
}
{
12
}, \,\,
DQ=\frac
{
-9 P^4+8 P R^3+9 Q^2
}
{
36
}, \,\,
DR
=
\frac
{
P^2 R+3 Q R
}
{
12
}.  
\end{equation*}
Then we have 
\begin{equation*}
D x=
D 
\left(
\frac{R^3}{P^3} 
\right)
=
\frac{R^3 \left(P^3-R^3\right)}{P^4},
\end{equation*}
\begin{equation*}
\frac{d c}{d x}
=\frac{d c}{d \tau} \frac{d \tau}{d x}=\frac{D R}{D x}
=
\frac{P^4 \left(P^2 R+3 Q R\right)}{12 R^3 \left(P^3-R^3\right)}, 
\end{equation*}
and 
\begin{equation*}
\frac{d^2 c}{d x^2}
=
\frac{d}{d \tau}
\left(
\frac{d c}{d x}
\right)
\frac{d x}{d\tau}
=
\frac{P^7 \left(-5 P^5-3 P^3 Q+8 P^2 R^3+12 Q R^3\right)}{36 R^5 \left(P^3-R^3\right)^2},
\end{equation*}
which imply that 
\begin{equation*}
x(1-x)\frac{d^2 c}{d x^2}+\frac13 (1-x) \frac{d c}{d x}-x \frac{d c}{d x}=
-\frac{P^3}{9 R^2}  
=
-\frac{P^3}{9 R^3} R=-\frac{1}{9x} c,
\end{equation*}
which proves the theorem. 
\end{proof}

\begin{theorem}
\label{thm:c-z-hypergeometric-E_2(q^3)}
{\it
For every $\tau\in\mathbb{H}^2$ set 
$
\displaystyle 
z(\tau)=
\frac
{
a^3(\tau)
}
{
c^3(\tau)
}.
$ 
Then the inverse function $\tau=\tau(z)$ yields $c=c(z)=c(\tau(z))$ and 
$c(z)$ satisfies the hypergeometric differential equations:
\begin{equation}
\label{eqn:c-z-hypergeometric}
z(1-z)
\frac{d^2 c}{dz^2}
+
\left(
\frac23-\frac53 z
\right)
\frac{dc}{dz} 
-
\frac{1}{9}
c
=0.
\end{equation}
}
\end{theorem}

\begin{proof}
The theorem can be proved by 
changing $\displaystyle x\to \frac{1}{x  }$ in Proposition \ref{prop:c-x-linear-E_2(q)}. 
\end{proof}

\begin{theorem}
\label{thm:c(q)-hypergeometric}
{\it
For every $\tau\in\mathbb{H}^2$ we have 
\begin{equation*}
c(\tau)
=
\left(
\frac
{
a^3(\tau)
}
{
c^3(\tau)
}
\right)^{-\frac13}
{ }_2 F_1
\left(
\frac13, 
\frac23; 
1; 
\frac
{
1
}
{
a^3(\tau)/c^3(\tau)
}
\right).
\end{equation*}
}
\end{theorem}

\begin{proof}
The formula follows from Theorem \ref{thm:a(q)-hypergeometric}. 
Note that 
$
\displaystyle
z^{-\frac13}
{ }_2 F_1
\left(
\frac13, 
\frac23; 
1; 
\frac
{
1
}
{
z
}
\right)
$ 
is a solution of the ODE (\ref{eqn:c-z-hypergeometric}). 
\end{proof}

\section{Heun's equations for modular forms of level 5 (1)  }
\label{sec:Heun-level 5-(1)}  

Throughout this section 
we set 
\begin{align*}
P(\tau)=&\cot \frac{2 \pi}{5} + 4 \sin \frac{4 \pi}{5} \sum_{n=1}^{\infty} (d_{1,5}(n)-d_{4,5}(n)) q^n
-4\sin \frac{2\pi}{5} \sum_{n=1}^{\infty} (d_{2,5}(n)-d_{3,5}(n)) q^n, \\
Q(\tau)=&\cot \frac{ \pi}{5} + 4 \sin \frac{2 \pi}{5} \sum_{n=1}^{\infty} (d_{1,5}(n)-d_{4,5}(n)) q^n
+4\sin \frac{4\pi}{5} \sum_{n=1}^{\infty} (d_{2,5}(n)-d_{3,5}(n)) q^n, \\
R(\tau)=&E_2(\tau)=1-24\sum_{n=1}^{\infty} \sigma_1(n) q^n=U(\tau). 
\end{align*}
In \cite{Matsuda2}, we proved that
\begin{align*}
\label{eqn-level5-E_2(q)}
DP=&\frac{-13P^3-39P^2Q+47PQ^2-9Q^3+2PR}{24},  \,\,
DQ=\frac{9P^3+47P^2Q+39PQ^2-13Q^3+2QR}{24}, \\
DR=&\frac{5P^4-15P^3Q-155P^2Q^2+15PQ^3+5Q^4+R^2}{12}, \,\,\,D=\frac{1}{2 \pi i} \frac{d}{d \tau}. 
\end{align*}

Moreover we set 
\begin{align*}
S(\tau)=&
\frac
{
\sqrt{250-110 \sqrt{5}}
}
{
5
}
V(\tau)
=
\frac
{
\sqrt{250-110 \sqrt{5}}
}
{
5
}
\prod_{n=1}^{\infty}
\frac
{
(1-q^n)^5
}
{
(1-q^{5n})^3
}
\left(
1+
\frac{1-\sqrt{5}}{2} 
q^n
+
q^{2n} 
\right)^5=-3 P+Q  \\
=&
\frac
{
\sqrt{250-110 \sqrt{5}}
}
{
5
}
\left\{
1
-
\frac52(1+\sqrt{5})
\sum_{n=1}^{\infty} (d_{1,5}(n)-d_{4,5}(n)) q^n 
+
\frac52
(7+3 \sqrt{5})
\sum_{n=1}^{\infty} (d_{2,5}(n)-d_{3,5}(n)) q^n
\right\},  
\end{align*}
and 
\begin{align*}
T(\tau)=&
\frac
{
\sqrt{250+110 \sqrt{5}}
}
{
5
}
W(\tau)
=
\frac
{
\sqrt{250+110 \sqrt{5}}
}
{
5
}
\prod_{n=1}^{\infty}
\frac
{
(1-q^n)^5
}
{
(1-q^{5n})^3
}
\left(
1+
\frac{1+\sqrt{5}}{2} 
q^n
+
q^{2n} 
\right)^5= P+3Q  \\
=&
\frac
{
\sqrt{250+110 \sqrt{5}}
}
{
5
}
\left\{
1
+
\frac52(-1+\sqrt{5})
\sum_{n=1}^{\infty} (d_{1,5}(n)-d_{4,5}(n)) q^n 
+
\frac52
(7-3 \sqrt{5})
\sum_{n=1}^{\infty} (d_{2,5}(n)-d_{3,5}(n)) q^n
\right\},  
\end{align*}
and 
\begin{equation*}
f(\tau)
=
\frac{1}{\tilde{f}(\tau)}
=
\frac{T}{S}
=
\frac{
\theta^{5}
\left[
\begin{array}{c}
 1 \\
\frac15
\end{array}
\right]
(0,\tau)
}
{
\theta^{5}
\left[
\begin{array}{c}
 1 \\
\frac35
\end{array}
\right](0,\tau)
}
=\frac{11+5\sqrt{5}}{2}
\prod_{n=1}^{\infty}
\frac{ \left(1+\frac{1+\sqrt{5}}{2}q^n+q^{2n} \right)^5 }{ \left(1+\frac{1-\sqrt{5}}{2}q^n+q^{2n} \right)^5 }, \,\,q=\exp(2\pi i \tau). 
\end{equation*}
Therefore it follows that  
\begin{align*}
D S=&
\frac{5 S^3+66 S^2 T-7 S T^2+20 S U}{240}, \,\,
DT=
\frac{-7 S^2 T-66 S T^2+5 T^3+20 T U}{240},    \\
DU=&
\frac{-S^4-12 S^3 T-14 S^2 T^2+12 S T^3-T^4+16 U^2}{192},
\end{align*}
and 
\begin{align*}
DV=&
\frac{5\left(25-11 \sqrt{5}\right) V^3+132 \sqrt{5} V^2 W-7 \left(25+11 \sqrt{5}\right) V W^2+50 VU }{600}  \\
DW=&
\frac{7 \left(-25+11 \sqrt{5}\right) V^2 W-132 \sqrt{5} V W^2+5 \left(25+11 \sqrt{5}\right) W^3 + 50 U W}{600} \\
DU=&
\frac{1}{120}\times 
\Big\{ 
\left( -123+55 \sqrt{5}\right) V^4+12 \left(11-5 \sqrt{5}\right) V^3 W-28 V^2 W^2  \\
&\hspace{45mm}+12 \left(11+5 \sqrt{5}\right) V W^3-\left(123+55 \sqrt{5}\right) W^4  
+10 U^2
\Big\}. 
\end{align*}

\subsection{On the modular function of level five}

\begin{theorem}
\label{thm:ODE-f-(1,1/5)-(1,3/5)}
{\it
For every $\tau\in\mathbb{H}^2$ set 
\begin{equation*}
f(\tau)=
\frac{
\theta^{5}
\left[
\begin{array}{c}
 1 \\
\frac15
\end{array}
\right]
(0,\tau)
}
{
\theta^{5}
\left[
\begin{array}{c}
 1 \\
\frac35
\end{array}
\right](0,\tau)
}
=\frac{11+5\sqrt{5}}{2}
\prod_{n=1}^{\infty}
\frac{ \left(1+\frac{1+\sqrt{5}}{2}q^n+q^{2n} \right)^5 }{ \left(1+\frac{1-\sqrt{5}}{2}q^n+q^{2n} \right)^5 }, \,\,q=\exp(2\pi i \tau). 
\end{equation*}
Then we have 
\begin{equation*}
\{f,\tau\}
+
\frac{f^4-12 f^3+134 f^2+12 f+1}{2 f^2 \left(f^2-11 f-1\right)^2}
\left(f^{\prime} \right)^2=0. 
\end{equation*}
}
\end{theorem}

\begin{proof}
We recall the ODE (\ref{eqn:ODE-j}) and the well-known property of the Schwartzian derivative, 
\begin{equation*}
\{j, \tau \}
=\{j, f \}
\left(
\frac{df}{d\tau}
\right)^2+\{f, \tau\} \,\,
\mathrm{and} \,\,
j=
\frac
{
\left(f^4-12 f^3+14 f^2 +12f+1\right)^3
}
{
f^5 (f^2-11f-1)
}.
\end{equation*}
For proof of the formula of $j(\tau),$ see Matsuda \cite{Matsuda2}. 
\par
Direct computation yields 
\begin{equation*}
\frac{dj}{d f}
=
\frac{5 \left(f^4-12 f^3+14 f^2+12 f+1\right)^2 \left(f^6-18 f^5+75 f^4+75 f^2+18 f+1\right)}{f^6 \left(f^2-11 f-1\right)^2}, 
\end{equation*}
\begin{align*}
\frac{d^2j}{d f^2}
=&
\frac{10 }{f^7 \left(f^2-11 f-1\right)^3}  \times  \\
&\times
\Big(2 f^{16}-96 f^{15}+1830 f^{14}-17395 f^{13}+83895 f^{12}-178866 f^{11}+90598 f^{10}  \\
&\hspace{8mm}
+34875 f^9-825 f^8-55100 f^7+328122 f^6+400419 f^5+159005 f^4+30030 f^3  \\
&\hspace{8mm}
+2970 f^2+149 f+3\Big),
\end{align*}
and
\begin{align*}
\frac{d^3j}{d f^3}
=&
\frac{30 }{f^8 \left(f^2-11 f-1\right)^4}  \times \\
&\times
\Big(2 f^{18}-108 f^{17}+2364 f^{16}-26584 f^{15}+159110 f^{14}-461196 f^{13}+425419 f^{12}  \\
&\hspace{8mm}
+191988 f^{11}+38652 f^{10}+80900 f^9-1371443 f^8+3744672 f^7+6973374 f^6  \\
&\hspace{8mm}
+3664956 f^5+918885 f^4+126564 f^3+9854 f^2+408 f+7\Big),
\end{align*}
which imply that 
\begin{align*}
\{j, f\}
=&
\frac{-12}{f^2 \left(f^4-12 f^3+14 f^2+12 f+1\right)^2 \left(f^6-18 f^5+75 f^4+75 f^2+18 f+1\right)^2} \times \\
&\times 
\Big( f^{20}-50 f^{19}+1020 f^{18}-10860 f^{17}+63875 f^{16}-201150 f^{15}+303500 f^{14}-214790 f^{13}  \\
&\hspace{8mm}
+201820 f^{12}-24450 f^{11}+2255152 f^{10}+24450 f^9+201820 f^8+214790 f^7+303500 f^6  \\
&\hspace{8mm}
+201150 f^5+63875 f^4+10860 f^3+1020 f^2+50 f+1\Big). 
\end{align*} 
Moreover we have 
\begin{align*}
&
\frac
{
j^2-2^4\cdot 3 \cdot 41 j+2^{15}\cdot 3^4
}
{
2 j^2(j-12^3)^2
}
(j^{\prime})^2
=
\frac
{
j^2-2^4\cdot 3 \cdot 41 j+2^{15}\cdot 3^4
}
{
2 j^2(j-12^3)^2
}
\left(
\frac{dj}{df}
\frac{d f}{d\tau}
\right)^2  \\
=&
\frac{25 }{2 f^2 \left(f^2+1\right)^2 \left(f^2-11 f-1\right)^2 \left(f^4-18 f^3+74 f^2+18 f+1\right)^2 \left(f^4-12 f^3+14 f^2+12 f+1\right)^2} \times   \\
&
\times
\Big(f^{24}-72 f^{23}+2244 f^{22}-39528 f^{21}+430626 f^{20}-2979240 f^{19}+12907780 f^{18}-33175560 f^{17}  \\
&\hspace{8mm}
+45886095 f^{16}-33834240 f^{15}+28236600 f^{14}-52867200 f^{13}+251679900 f^{12}+52867200 f^{11}  \\
&\hspace{8mm}
+28236600 f^{10}+33834240 f^9+45886095 f^8+33175560 f^7+12907780 f^6+2979240 f^5  \\
&\hspace{8mm}
+430626 f^4+39528 f^3+2244 f^2+72 f+1\Big) \left( f^{\prime} \right)^2. 
\end{align*}
Therefore it follows that 
\begin{equation*}
\{f,\tau\}
+
\frac{f^4-12 f^3+134 f^2+12 f+1}{2 f^2 \left(f^2-11 f-1\right)^2}
\left(f^{\prime} \right)^2=0. 
\end{equation*}
\end{proof}

\subsection{Heun's equations for $S(\tau)$ and $V(\tau)$}

\begin{theorem}
\label{thm:S-Heun-level5-f}
{\it
For every $\tau\in\mathbb{H}^2$ set 
\begin{equation*}
\displaystyle 
f(\tau)
=\frac{11+5\sqrt{5}}{2}
\prod_{n=1}^{\infty}
\frac{ \left(1+\frac{1+\sqrt{5}}{2}q^n+q^{2n} \right)^5 }{ \left(1+\frac{1-\sqrt{5}}{2}q^n+q^{2n} \right)^5 }, \,\,q=\exp(2\pi i \tau). 
\end{equation*}
Then the inverse function $\tau=\tau(f)$ yields $S(f)=S(\tau(f))$ and $S(f)$ satisfies the following differential equation:
\begin{equation*}
\frac{d^2 S}{d f^2}
+
\left\{
\frac{1}{f}
+
\frac{1}{f-  \frac{11+5 \sqrt{5}}{2}  }
+
\frac{1}{f-  \frac{11-5 \sqrt{5}}{2}  }
\right\}
\frac{d S}{d f}
+
\frac
{
f-3
}
{
f(f-  \frac{11+5 \sqrt{5}}{2})(f-\frac{11-5 \sqrt{5}}{2}   )
}
S=0.
\end{equation*}
}
\end{theorem}

\begin{proof}
We first have 
\begin{equation*}
Df=
D
\left(
\frac{T}{S}
\right)=
-\frac{T \left(S^2+11 S T-T^2\right)}{20 S}. 
\end{equation*}
We next obtain 
\begin{equation*}
\frac{d S}{d f}
=
\frac{d S}{d\tau} \frac{d \tau}{d f}
=
\frac
{
D S
}
{
D f
}
=
-\frac{S \left(5 S^3+66 S^2 T-7 S T^2+20 S U\right)}{12 T \left(S^2+11 S T-T^2\right)},
\end{equation*}
and
\begin{equation*}
\frac{d^2 S}{d f^2}
=
\frac{d}{d \tau}
\left(
\frac{d S}{d f}
\right)
\frac{d\tau}{d f}
=
\frac{S^3 \left(5 S^4+140 S^3 T+1046 S^2 T^2+20 S^2 U-184 S T^3+440 S T U+9 T^4-60 T^2 U\right)}{12 T^2 \left(S^2+11 S T-T^2\right)^2}.
\end{equation*}
Solving the algebraic equation $f^2-11 f-1=0$ yields 
$
\displaystyle
\alpha=\frac{11+5\sqrt{5}}{2} 
$ 
and 
$
\displaystyle
\beta=\frac{11-5\sqrt{5}}{2}.
$ 
\par
Therefore it follows that  
\begin{align*}
&
f\left(f-\alpha\right) \left(f-\beta \right)\frac{d^2 S}{d f^2}
+
\left(f-\alpha\right) \left(f-\beta \right)\frac{d S}{d f}
+
f\left(f-\alpha \right)\frac{d S}{d f}  
+
f\left(f-\beta \right)\frac{d S}{d f}  \\
=&
3 S-T
=S
\left(
3-\frac{T}{S}
\right)
=
S(3-f),
\end{align*}
which proves the theorem. 
\end{proof}

\begin{corollary}
\label{coro:V-Heun-level5-v}
{\it 
For every $\tau\in\mathbb{H}^2$ set 
$
\displaystyle
v(\tau)=
\frac{2f-(11+5 \sqrt{5} ) }{5\sqrt{5} (11+5 \sqrt{5 })   }
=
\frac{1}{5 \sqrt{5}}
\left(
\frac{W-V}{V}
\right). 
$
Then the inverse function $\tau=\tau(v)$ yields $V(v)=V(\tau(v))$ and $V(v)$ satisfies 
Heun's differential equation,
\begin{equation*}
\frac{d^2 V}{d v^2}
+
\left\{
\frac{1}{v}
+
\frac{1}{v+\frac{ \sqrt{5}  }{25} }
+
\frac{1}{v-  \frac{11-5 \sqrt{5}}{2}  }
\right\}
\frac{d V}{d v}
+
\frac
{
v-
\frac{15-7\sqrt{5}}{10}
}
{
v(v+\frac{ \sqrt{5} }{ 25 })(v-\frac{11-5 \sqrt{5}}{2}   )
}
V=0.
\end{equation*}
}
\end{corollary}

\begin{corollary}
\label{coro:V-Heun-level5-v-sporadic}
{\it 
For every $\tau\in\mathbb{H}^2$ set 
$
\displaystyle
v(\tau)=
\frac{1}{5 \sqrt{5}}
\left(
\frac{W-V}{V}
\right)
$ 
and 
let 
$
\left\{
a_5(n)
\right\}_{n\geq 0}
$ 
be the sequence defined by the initial condition $a_5(0)=1$ and recurrence relation 
\begin{align}
&
(n+1)^2 a_5(n+1)
+\left\{\frac{\left(11+15 \sqrt{5}\right)}{2}  n (n+1)+\frac{5\left(1+\sqrt{5}\right)}{2} \right\} a_5(n) \notag  \\
&\hspace{80mm}
+
\frac{5\left(25+11 \sqrt{5}\right)}{2}  n^2 a_5(n-1)
=0.   \label{eqn:recurrence-a-(1,1/5)-(1,3/5)}
\end{align} 
Then we have 
\begin{equation*}
V(\tau)=\sum_{n=0}^{\infty} a_5(n) v^n. 
\end{equation*}
}
\end{corollary}

\begin{corollary}
\label{coro:V-Heun-level5-g}
{\it 
For every $\tau\in\mathbb{H}^2$ set 
\begin{equation*}
g(\tau)=
q
\prod_{n=1}^{\infty}
\frac
{
(1-q^{5n-1})^5 (1-q^{5n-4})^5 
}
{
(1-q^{5n-2})^5 (1-q^{5n-3})^5 
}, \,\,\, q=\exp(2\pi i \tau).
\end{equation*}
Then the inverse function $\tau=\tau(g)$ yields $V(g)=V(\tau(g))$ and $V(g)$ satisfies 
Heun's differential equation,
\begin{equation*}
\frac{d^2 V}{d g^2}
+
\left\{
\frac{1}{g}
+
\frac{1}{g+\frac{ 11+5\sqrt{5}  }{2} }
-
\frac{1}{g+ \frac{11-5 \sqrt{5}}{2}  }
\right\}
\frac{d V}{d g}
-
\frac
{
5(1-\sqrt{5})
}
{
2
}
\frac
{
\left\{
g+
(-2+\sqrt{5}
)
\right\}
}
{
g
\left(
g+\frac{ 11+5\sqrt{5}  }{2}
\right)
\left(
g+ \frac{11-5 \sqrt{5}}{2}
\right)^2
}
V=0.
\end{equation*}
}
\end{corollary}

\begin{proof}
From Farkas and Kra \cite[p. 262]{Farkas-Kra} 
recall the following formula:
$
\displaystyle
g=
\frac
{
\frac{-11+5\sqrt{5}}{2} f-1
}
{
f
+
\frac{-11+5\sqrt{5}}{2}
}.
$
\end{proof}

\begin{corollary}
\label{coro:V-Heun-level5-g-sporadic}
{\it 
For every $\tau\in\mathbb{H}^2$ set 
\begin{equation*}
g(\tau)=
q
\prod_{n=1}^{\infty}
\frac
{
(1-q^{5n-1})^5 (1-q^{5n-4})^5 
}
{
(1-q^{5n-2})^5 (1-q^{5n-3})^5 
}, \,\,\, q=\exp(2\pi i \tau).
\end{equation*}
and 
let 
$
\left\{
s_5(n)
\right\}_{n\geq 0}
$ 
be the sequence defined by the initial condition $s_5(0)=1$ and recurrence relation 
\begin{align}
&(n+1)^2 s_5(n+1)
+\left\{-\frac{33+5 \sqrt{5}}{2}  n^2+5 \sqrt{5} n+\frac{5\left(1+\sqrt{5}\right)}{2}   \right\} s_5(n)   \notag  \\
&\hspace{15mm}+\left\{ \frac{119+55 \sqrt{5}}{2}  n^2-\frac{33 \left(11+5 \sqrt{5}\right)  }{2}  n+\frac{ 279+125 \sqrt{5}  }{2}   \right\} s_5(n-1)  \notag \\
&\hspace{25mm}
+\frac{11+5 \sqrt{5}}{2}  (n-2)^2 s_5(n-2)
=0.
\end{align}
Then we have 
\begin{equation*}
V(\tau)=\sum_{n=0}^{\infty} s_5(n) g^n. 
\end{equation*}
}
\end{corollary}

\subsection{Heun's equations for $T(\tau)$ and $W(\tau)$}

\begin{proposition}
\label{prop:T-Heun-level 5-f}
{\it
For every $\tau\in\mathbb{H}^2$ set 
\begin{equation*}
\displaystyle 
f(\tau)
=\frac{11+5\sqrt{5}}{2}
\prod_{n=1}^{\infty}
\frac{ \left(1+\frac{1+\sqrt{5}}{2}q^n+q^{2n} \right)^5 }{ \left(1+\frac{1-\sqrt{5}}{2}q^n+q^{2n} \right)^5 }, \,\,q=\exp(2\pi i \tau). 
\end{equation*}
Then the inverse function $\tau=\tau(f)$ yields $T(f)=T(\tau(f))$ and $T(f)$ satisfies the following differential equation:
\begin{equation*}
\frac{d^2 T}{d f^2}
+
\left\{
-
\frac{1}{f}
+
\frac{1}{f-  \frac{11+5 \sqrt{5}}{2}  }
+
\frac{1}{f-  \frac{11-5 \sqrt{5}}{2}  }
\right\}
\frac{d T}{d f}
-
\frac
{
3f+1
}
{
f^2(f-  \frac{11+5 \sqrt{5}}{2})(f-\frac{11-5 \sqrt{5}}{2}   )
}
T=0.
\end{equation*}
}
\end{proposition}

\begin{proof}
We use the same notation as Theorem \ref{thm:S-Heun-level5-f}. 
We first have 
\begin{equation*}
\frac{d T}{d f}
=
\frac{d T}{d \tau}
\frac{d \tau}{d f}
=\frac{ D T}{D f}
=
\frac{S \left(7 S^2+66 S T-5 T^2-20 U\right)}{12 \left(S^2+11 S T-T^2\right)}, 
\end{equation*}
and 
\begin{equation*}
\frac{d^2 T}{d f^2}
=
\frac{d}{d \tau}
\left(
\frac{d T}{d f}
\right)
\frac{d \tau}{d f}
=
-\frac{S^2 \left(5 S^4+102 S^3 T+382 S^2 T^2+20 S^2 U-102 S T^3+5 T^4+20 T^2 U\right)}{12 T \left(S^2+11 S T-T^2\right)^2},
\end{equation*}
which imply that 
\begin{align*}
&
f(f-\alpha)(f-\beta)\frac{d^2 T}{d f^2}
-(f-\alpha)(f-\beta)\frac{d T}{d f}
+f(f-\alpha)+f(f-\beta)\frac{d T}{d f}   \\
=&
S+3 T
=
T
\left(
3+\frac{S}{T}
\right)
=
T
\left(
3+\frac{1}{f}
\right)
=
\frac{3 f+1}{f} T,
\end{align*}
which proves the proposition. 
\end{proof}

\begin{theorem}
\label{thm:T-Heun-f}
{\it
For every $\tau\in\mathbb{H}^2$ set 
\begin{equation*}
\displaystyle 
\tilde{f}(\tau)
=\frac{-11+5\sqrt{5}}{2}
\prod_{n=1}^{\infty}
\frac
{ \left(1+\frac{1-\sqrt{5}}{2}q^n+q^{2n} \right)^5 }
{ \left(1+\frac{1+\sqrt{5}}{2}q^n+q^{2n} \right)^5 }, \,\,q=\exp(2\pi i \tau). 
\end{equation*}
Then the inverse function $\tau=\tau( \tilde{f})$ yields $T(\tilde{f})=T(\tau(\tilde{f}))$ and $T(\tilde{f})$ satisfies the following differential equation:
\begin{equation*}
\frac{d^2 T}{d \tilde{f}^2}
+
\left\{
\frac{1}{ \tilde{f}  }
+
\frac{1}{  \tilde{f}-  \frac{-11+5 \sqrt{5}}{2}  }
+
\frac{1}{ \tilde{f}-  \frac{-11-5 \sqrt{5}}{2}  }
\right\}
\frac{d T}{d \tilde{f}  }
+
\frac
{
\tilde{f}+3
}
{
\tilde{f}( \tilde{f}-  \frac{-11+5 \sqrt{5}}{2})( \tilde{f}-\frac{-11-5 \sqrt{5}}{2}   )
}
T=0.
\end{equation*}
}
\end{theorem}

\begin{proof}
The theorem can be proved by changing 
$
\displaystyle 
f
\to 
\frac{1}{f}
$ 
in Proposition \ref{prop:T-Heun-level 5-f}. 
\end{proof}

\begin{corollary}
\label{coro:W-Heun-level5-w}
{\it 
For every $\tau\in\mathbb{H}^2$ set 
$
\displaystyle
w(\tau)=
\frac{ 2\tilde{f}-(-11+5 \sqrt{5})  }{-5\sqrt{5 } (-11+5 \sqrt{5} )   }
=
\frac{1}{ -5\sqrt{5}  }
\left(
\frac{V-W}{W}
\right). 
$
Then the inverse function $\tau=\tau(w)$ yields $W( w)=W(\tau( w))$ and $W(w)$ satisfies 
Heun's differential equation,
\begin{equation*}
\frac{d^2 W}{d w^2}
+
\left\{
\frac{1}{ w }
+
\frac{1}{ w-\frac{ \sqrt{5} }{25} }
+
\frac{1}{ w-  \frac{11+5 \sqrt{5}}{2}  }
\right\}
\frac{d W}{d w}
+
\frac
{
w-
\frac{15+7\sqrt{5}}{10}
}
{
w( w-\frac{ \sqrt{5} }{25})( w-\frac{11+5 \sqrt{5}}{2}  )
}
W=0.
\end{equation*}
}
\end{corollary}

\begin{corollary}
\label{coro:W-Heun-level5-w-sporadic}
{\it 
For every $\tau\in\mathbb{H}^2$ set 
$
\displaystyle
w(\tau)=
\frac{1}{-5 \sqrt{5}}
\left(
\frac{V-W}{W}
\right)
$ 
and 
let 
$
\left\{
b_5(n)
\right\}_{n\geq 0}
$ 
be the sequence defined by the initial condition $b_5(0)=1$ and recurrence relation 
\begin{align}
&
(n+1)^2 b_5(n+1)
+\left\{\frac{\left(11-15 \sqrt{5}\right)}{2}  n (n+1)+\frac{5\left(1-\sqrt{5}\right)}{2} \right\} b_5(n) \notag  \\
&\hspace{80mm}
+
\frac{5\left(25-11 \sqrt{5}\right)}{2}  n^2 b_5(n-1)
=0.   \label{eqn:recurrence-b-(1,1/5)-(1,3/5)}
\end{align} 
Then we have 
\begin{equation*}
W(\tau)=\sum_{n=0}^{\infty} b_5(n) w^n. 
\end{equation*}
}
\end{corollary}

\begin{corollary}
\label{coro:W-Heun-level5-g}
{\it 
For every $\tau\in\mathbb{H}^2$ set 
\begin{equation*}
g(\tau)=
q
\prod_{n=1}^{\infty}
\frac
{
(1-q^{5n-1})^5 (1-q^{5n-4})^5 
}
{
(1-q^{5n-2})^5 (1-q^{5n-3})^5 
}, \,\,\, q=\exp(2\pi i \tau).
\end{equation*}
Then the inverse function $\tau=\tau(g)$ yields $W(g)=W(\tau(g))$ and $W(g)$ satisfies 
Heun's differential equation,
\begin{equation*}
\frac{d^2 W}{d g^2}
+
\left\{
\frac{1}{g}
+
\frac{1}{g+\frac{ 11-5\sqrt{5}  }{2} }
-
\frac{1}{g+ \frac{11+5 \sqrt{5}}{2}  }
\right\}
\frac{d W}{d g}
-
\frac
{
5(1+\sqrt{5})
}
{
2
}
\frac
{
\left\{
g-
(2+\sqrt{5}
)
\right\}
}
{
g
\left(
g+\frac{ 11-5\sqrt{5}  }{2}
\right)
\left(
g+ \frac{11+5 \sqrt{5}}{2}
\right)^2
}
W=0.
\end{equation*}
}
\end{corollary}

\begin{proof}
The corollary can be proved in the same way as 
Corollary \ref{coro:V-Heun-level5-g}. 
\end{proof}

\begin{corollary}
\label{coro:W-Heun-level5-g-sporadic}
{\it 
For every $\tau\in\mathbb{H}^2$ set 
\begin{equation*}
g(\tau)=
q
\prod_{n=1}^{\infty}
\frac
{
(1-q^{5n-1})^5 (1-q^{5n-4})^5 
}
{
(1-q^{5n-2})^5 (1-q^{5n-3})^5 
}, \,\,\, q=\exp(2\pi i \tau).
\end{equation*}
and 
let 
$
\left\{
t_5(n)
\right\}_{n\geq 0}
$ 
be the sequence defined by the initial condition $s_5(0)=1$ and recurrence relation 
\begin{align}
&(n+1)^2 t_5(n+1)
+\left\{-\frac{33-5 \sqrt{5}}{2}  n^2-5 \sqrt{5} n+\frac{5\left(1-\sqrt{5}\right)}{2}   \right\} t_5(n)   \notag  \\
&\hspace{15mm}+\left\{ \frac{119-55 \sqrt{5}}{2}  n^2-\frac{33 \left(11-5 \sqrt{5}\right)  }{2}  n+\frac{ 279-125 \sqrt{5}  }{2}   \right\} t_5(n-1)  \notag \\
&\hspace{25mm}
+\frac{11-5 \sqrt{5}}{2}  (n-2)^2 t_5(n-2)
=0.
\end{align}
Then we have 
\begin{equation*}
W(\tau)=\sum_{n=0}^{\infty} t_5(n) g^n. 
\end{equation*}
}
\end{corollary}

\section{Heun's equations for modular forms of level 5 (2)  }
\label{sec:Heun-level 5-(2)}

Throughout this section we set 
\begin{align*}
P(\tau)=&1+ 10 \sum_{n=1}^{\infty} (d_{2,5}(n)-d_{3,5}(n)) q^n, \,\,
Q(\tau)=3 + 10  \sum_{n=1}^{\infty} (d_{1,5}(n)-d_{4,5}(n)) q^n, \\
R(\tau)=&E_2(5\tau)=1-24\sum_{n=1}^{\infty} \sigma_1(n) q^{5n}=U(\tau). 
\end{align*}
In \cite{Matsuda2}, we proved that
\begin{align*}
DP=&\frac{13P^3+39P^2Q-47PQ^2+9Q^3+50PR}{120},  \,\,
DQ=\frac{-9P^3-47P^2Q-39PQ^2+13Q^3+50QR}{120}, \\
DR=&\frac{P^4-3P^3Q-31P^2Q^2+3PQ^3+Q^4+125R^2}{300}, \,\,\,D=\frac{1}{2 \pi i} \frac{d}{d \tau}. 
\end{align*}
\par
Moreover we set 
\begin{align*}
S(\tau)=&
q
\prod_{n=1}^{\infty} 
\frac
{
(1-q^n)^2
}
{
(1-q^{5n-2})^5(1-q^{5n-3})^5    
}
=-\frac{3}{10}P+\frac{1}{10} Q   \\
=&
\sum_{n=1}^{\infty} (d_{1,5}(n)-d_{4,5}(n)) q^n
-3
\sum_{n=1}^{\infty} (d_{2,5}(n)-d_{3,5}(n)) q^n,  
\end{align*}
and 
\begin{align*}
T(\tau)=&
\prod_{n=1}^{\infty} 
\frac
{
(1-q^n)^2
}
{
(1-q^{5n-1})^5(1-q^{5n-4})^5    
}
=
\frac{1}{10} P+\frac{3}{10} Q  \\
=&
1+
3
\sum_{n=1}^{\infty} (d_{1,5}(n)-d_{4,5}(n)) q^n
+
\sum_{n=1}^{\infty} (d_{2,5}(n)-d_{3,5}(n)) q^n,
\end{align*}
which imply that 
\begin{align*}
DS=&\frac{-5 S^3-66 S^2 T+7 S T^2+5 S U}{12},  \,\, 
DT=
\frac{7 S^2 T+66 S T^2-5 T^3+5 T U}{12},  \\   
DU=&
\frac{-5 S^4-60 S^3 T-70 S^2 T^2+60 S T^3-5 T^4+5 U^2}{12}. 
\end{align*}

\subsection{On the modular function of level five}

\begin{theorem}
\label{thm:ODE-g-(1/5,1)-(3/5,1)}
{\it
For every $\tau\in\mathbb{H}^2$ set 
\begin{equation*}
g(\tau)=
q
\prod_{n=1}^{\infty}
\frac
{
(1-q^{5n-1})^5 (1-q^{5n-4})^5 
}
{
(1-q^{5n-2})^5 (1-q^{5n-3})^5 
}, \,\,\, q=\exp(2\pi i \tau).
\end{equation*}
Then we have 
\begin{equation*}
\{g,\tau\}
+
\frac{g^4+12 g^3+134 g^2-12 g+1}{2 g^2 \left(g^2+11 g-1\right)^2}
\left(g^{\prime} \right)^2=0. 
\end{equation*}
}
\end{theorem}

\begin{proof}
The theorem can be proved in the same way as Theorem \ref{thm:ODE-f-(1,1/5)-(1,3/5)}. 
Recall the formula, 
\begin{equation*}
j=
\frac{\left(g^4-228 g^3+494 g^2+228 g+1\right)^3}{g \left(-g^2-11 g+1\right)^5}. 
\end{equation*}
For the proof of this formula, see Cooper \cite[p. 322]{Cooper} or Matsuda \cite{Matsuda2}. 
\end{proof}

\subsection{Heun's equation for the modular form of level five}

\begin{theorem}(Beukers \cite{Beukers}) 
\label{thm:T-Heun-level 5-g}
{\it
For every $\tau\in\mathbb{H}^2$ set 
\begin{equation*}
g(\tau)=
q
\prod_{n=1}^{\infty}
\frac
{
(1-q^{5n-1})^5 (1-q^{5n-4})^5 
}
{
(1-q^{5n-2})^5 (1-q^{5n-3})^5 
}, \,\,\, q=\exp(2\pi i \tau).
\end{equation*}
Then the inverse function $\tau=\tau( g )$ yields $T(g)=T(\tau(g))$ and $T(g)$ satisfies Heun's equation:
\begin{equation*}
\frac{d^2 T}{d g^2}
+
\left\{
\frac{1}{ g }
+
\frac{1}{ g-  \frac{-11+5 \sqrt{5}}{2}  }
+
\frac{1}{  g-  \frac{-11-5 \sqrt{5}}{2}  }
\right\}
\frac{d T}{d g}
+
\frac
{
g+3
}
{
g( g-  \frac{-11+5 \sqrt{5}}{2})( g-\frac{-11-5 \sqrt{5}}{2}   )
}
T=0.
\end{equation*}
}
\end{theorem}

\begin{proof}
We first have 
\begin{equation*}
D g=
D
\left(
\frac{S}{T}
\right)
=
-\frac{S \left(S^2+11 S T-T^2\right)}{T}. 
\end{equation*}
Then it follows that 
\begin{equation*}
\frac{d T}{d g}
=
\frac{d T}{ d\tau} \frac{d \tau}{dg}
=
\frac{DT}{D g}
=
-\frac{T \left(7 S^2 T+66 S T^2-5 T^3+5 T U\right)}{12 S \left(S^2+11 S T-T^2\right)},
\end{equation*}
and 
\begin{equation*}
\frac{d^2 T}{d g^2}
=
\frac{d }{d \tau}
\left(
\frac{d T}{d g}
\right)
\frac{d \tau}{d g} 
=
\frac{T^3 \left(9 S^4+184 S^3 T+1046 S^2 T^2+15 S^2 U-140 S T^3+110 S T U+5 T^4-5 T^2 U\right)}{12 S^2 \left(S^2+11 S T-T^2\right)^2},
\end{equation*}
which imply that 
\begin{align*}
&
g(g-\alpha)(g-\beta) \frac{d^2 T}{d g^2}
+(g-\alpha)(g-\beta) \frac{d T}{d g}
+g(g-\alpha) \frac{d T}{d g}
+g(g-\beta) \frac{d T}{d g}  \\
=&
-S-3 T
=
-T
\left(
\frac{S}{T}
+3
\right)
=
-T(g+3),
\end{align*}
which proves the theorem. 
\end{proof}

\begin{corollary}
\label{coro:T-Heun-level5-g}
{\it 
For every $\tau\in\mathbb{H}^2$ set 
$
\displaystyle
G(\tau)=
\frac{2g}{-11+5\sqrt{5}}. 
$
The inverse function $\tau=\tau(G)$ yields $T(G)=T(\tau(G))$ and $T(G)$ satisfies 
Heun's differential equation,
\begin{equation*}
\frac{d^2 T}{d G^2}
+
\left\{
\frac{1}{G}
+
\frac{1}{G-1 }
+
\frac{1}{G-  \frac{-123-55 \sqrt{5}}{2}  }
\right\}
\frac{d T}{d G}
+
\frac
{
G+
\frac{3(11+5\sqrt{5} ) }{2}
}
{
G(G-1)(G-\frac{-123-55 \sqrt{5}}{2}   )
}
T=0.
\end{equation*}
}
\end{corollary}

\section{Heun's equations for modular forms of level 6 (1)}
\label{sec:Heun-f-(1,2/3)}

Throughout this section we set 
\begin{align*}
P(\tau)=&
\frac
{
\eta(2\tau) \eta^6(3\tau)
}
{
\eta^2(\tau) \eta^3(6\tau)
}
=1+2\sum_{n=1}^{\infty} (d_{1,6}(n)+d_{2,6}(n)-d_{4,6}(n)-d_{5,6}(n)) q^n,   \\
Q(\tau)
=
&
\frac
{
\eta^6(\tau) \eta(6\tau)
}
{
\eta^3(2\tau) \eta^2(3\tau)
}
=
1-6\sum_{n=1}^{\infty} (d_{1,6}(n)-d_{5,6}(n)) q^n+18 \sum_{n=1}^{\infty} ( d_{2,6}(n)-d_{4,6}(n) )q^n,  \\
R=&E_2(\tau), 
\end{align*}
and 
\begin{equation*}
f(\tau)=\frac{P(\tau)}{Q(\tau)}
=
\frac
{
\eta^4(2\tau) \eta^8(3\tau)
}
{
\eta^8(\tau) \eta^4(6\tau)
}.
\end{equation*}

In \cite{Matsuda4}, we proved that 
\begin{align*}
\label{eqn-level6-E_2(q)}
DP=&\frac{27 P^3-36P^2 Q+5P Q^2+4 P R}{48},  \,\,
DQ=\frac{ -27 P^2 Q+24 P Q^2-Q^3+4 Q R }{48}, \\
DR=&\frac{ -729 P^4+972 P^3 Q-270 P^2 Q^2+12 P Q^3-Q^4+16 R^2}{192}, \,\, \,D= \frac{1}{2 \pi i} \frac{d}{d\tau}.
\end{align*}

\subsection{On the modular function of level six}

\begin{theorem}
\label{thm:ODE-f-(1,2/3)}
{\it
For every $\tau\in\mathbb{H}^2$ set 
\begin{equation*}
f(\tau)=
\frac
{
\eta^4(2\tau) \eta^8(3\tau)
}
{
\eta^8(\tau) \eta^4(6\tau)
}.
\end{equation*}
Then we have 
\begin{equation*}
\{f,\tau\}
+
\frac{81 f^4-108 f^3+102 f^2-12 f+1}{2 (1-9 f)^2 (f-1)^2 f^2}
(f^{\prime})^2=0.
\end{equation*}
}
\end{theorem}

\begin{proof}
We recall the ODE (\ref{eqn:ODE-j}) 
and the well-known property of the Schwartzian derivative, 
\begin{equation*}
\{j, \tau \}
=\{j, f \}
\left(
\frac{df}{d\tau}
\right)^2+\{f, \tau\} \,\,
\mathrm{and} \,\,
j=
\frac{(3 f-1)^3 \left(243 f^3-243 f^2+9 f-1\right)^3}{ f^2  (f-1)  (9 f-1)^3}. 
\end{equation*}
For the proof of the formula of $j(\tau)$ see Matsuda \cite{Matsuda4}. 
Direct computation yields 
\begin{equation*}
\frac{dj}{d f}
=
\frac{2 \left(729 f^4-972 f^3+270 f^2-12 f+1\right)^2 \left(19683 f^6-39366 f^5+24057 f^4-4860 f^3+405 f^2+18 f-1\right)}{(1-9 f)^4 (f-1)^2 f^3}, 
\end{equation*}
\begin{align*}
\frac{d^2j}{d f^2}
=&
\frac{2 (3 f-1) }{(f-1)^3 f^4 (9 f-1)^5}  \times \\
&\times
\Big(156905298045 f^{15}-784526490225 f^{14}+1639950929937 f^{13}-1861555449645 f^{12}   \\
&\hspace{8mm}
+1252530440937 f^{11}-516737621853 f^{10}+133164765693 f^9-21863305593 f^8+2275295751 f^7  \\
&\hspace{8mm}
-148035843 f^6+4638627 f^5-360855 f^4+28755 f^3-2799 f^2+143
   f-3\Big)
\end{align*}
and
\begin{align*}
\frac{d^3j}{d f^3}
=&
\frac{24}{(1-9 f)^6 (f-1)^4 f^5}  \times \\
&\times
 \Big(1412147682405 f^{18}-8472886094430 f^{17}+21998122785909 f^{16}-32438717543970 f^{15}   \\
&\hspace{8mm}
+30039809876082 f^{14}-18300796872498 f^{13}+7510112698482 f^{12}-2103967478826 f^{11}   \\
&\hspace{8mm}
+405174161340 f^{10}-53490560166 f^9+4755176604 f^8-271927206 f^7+6192126 f^6  \\
&\hspace{8mm}
+1562490 f^5-203778 f^4+22642 f^3-1553 f^2+60 f-1\Big), 
\end{align*}
which imply that 
\begin{align*}
\{j, f\}
=&
\frac{-3 }{2 (1-9 f)^2 (1-3 f)^2 f^2 \left(-243 f^3+243 f^2-9 f+1\right)^2 }  \times \\
&\times
\frac{1}{\left(19683 f^6-39366 f^5+24057 f^4-4860 f^3+405 f^2+18 f-1\right)^2} \\
&\times 
\Big(194567119829443305 f^{22}-1167402718976659830 f^{21}+3086513961230883159 f^{20} \\
&\hspace{8mm}
-4751509792895399700 f^{19}+4749107729687628795 f^{18}-3260442825111513870 f^{17}  \\
&\hspace{8mm}
+1590064639627100085 f^{16}-562741730108061552 f^{15}+146299260016157418 f^{14}  \\
&\hspace{8mm}
-27993124770927372 f^{13}+3894278695217046 f^{12}-379532607843960 f^{11}   \\
&\hspace{8mm}
+24947612364294 f^{10}-1601012779308 f^9+249827934042 f^8-38521048176 f^7   \\
&\hspace{8mm}
+3390495165 f^6-149332734 f^5+2531331 f^4-6804 f^3+1647 f^2-70 f+1\Big). 
\end{align*}
Moreover we have 
\begin{align*}
&
\frac
{
j^2-2^4\cdot 3 \cdot 41 j+2^{15}\cdot 3^4
}
{
2 j^2(j-12^3)^2
}
(j^{\prime})^2
=
\frac
{
j^2-2^4\cdot 3 \cdot 41 j+2^{15}\cdot 3^4
}
{
2 j^2(j-12^3)^2
}
\left(
\frac{dj}{df}
\frac{d f}{d\tau}
\right)^2  \\
=&
\frac{2 }{(f-1)^2 f^2 (3 f-1)^2 (9 f-1)^2 \left(27 f^2-18 f-1\right)^2 } \times  \\
&\times
\frac
{1}{\left(243 f^3-243 f^2+9 f-1\right)^2 \left(729 f^4-972 f^3+270 f^2-36 f+1\right)^2} \times  \\
&\times
\Big(150094635296999121 f^{24}-1200757082375992968 f^{23}+4336067241913307940 f^{22}  \\
&\hspace{8mm}
-9354045913324093368 f^{21}+13454161917856933554 f^{20}-13647562321271173848 f^{19}   \\
&\hspace{8mm}
+10083015752190035076 f^{18}-5533117719052524264 f^{17}+2282756906629011087 f^{16}   \\
&\hspace{8mm}
-712960844770437840 f^{15}+168892648959528072 f^{14}-30199063458850992 f^{13}  \\
&\hspace{8mm}
+4014222500231676 f^{12}-385525591473456 f^{11}+26847128113128 f^{10}-1926192551760 f^9  \\
&\hspace{8mm}
+254289026943 f^8-33270096168 f^7+2669274324 f^6-110265624 f^5+1786482 f^4  \\
&\hspace{8mm}
-11640 f^3+1812 f^2-72 f+1\Big). 
\end{align*}
Therefore it follows that 
\begin{equation*}
\{f,\tau\}
+
\frac{81 f^4-108 f^3+102 f^2-12 f+1}{2 (1-9 f)^2 (f-1)^2 f^2}
\left(f^{\prime} \right)^2=0. 
\end{equation*}
\end{proof}

\subsection{Heun's equation for $P(\tau)$}

\begin{proposition}
\label{prop:P-f-linear-E_2(q)}
{\it
For every $\tau\in\mathbb{H}^2$ set 
$
\displaystyle 
f(\tau)=
\frac
{
\eta^4(2\tau) \eta^8(3\tau)
}
{
\eta^8(\tau) \eta^4(6\tau)
}.
$ 
Then the inverse function $\tau=\tau(f)$ yields $P(f)=P(\tau(f))$ and 
$P(f)$ satisfies the following differential equation:
\begin{equation}
\frac{d^2 P}{d f^2}
+
\left\{
-
\frac{1}{f}
+
\frac{1}{f-1}
+
\frac{1}{f-\frac19}
\right\}
\frac{dP}{df}
-
\frac
{
f-\frac13
}
{
3 f^2(f-1)(f-\frac19)
}
P=0. 
\end{equation}
}
\end{proposition}

\begin{proof}
We first have 
\begin{equation*}
Df=D\left(  \frac{P}{Q} \right)
=
\frac{P \left(9 P^2-10 P Q+Q^2\right)}{8 Q}. 
\end{equation*}
\par
We next obtain 
\begin{equation*}
\frac{d P}{d f}
=
\frac{d P}{d \tau}
\frac{d \tau}{d f}
=
\frac
{
D P
}
{
D f
}
=
\frac{Q \left(27 P^2-36 P Q+5 Q^2+4 R\right)}{6 \left(9 P^2-10 P Q+Q^2\right)},
\end{equation*}
and 
\begin{equation*}
\frac{d^2 P}{d f^2}
=
\frac{d}{d \tau}
\left(
\frac{d P}{d f}
\right)
\frac{d\tau}{d f} 
=
-\frac{Q^2 (3 P-Q) \left(81 P^3-135 P^2 Q+39 P Q^2+12 P R-Q^3+4 Q R\right)}{6 P \left(9 P^2-10 P Q+Q^2\right)^2}. 
\end{equation*}
\par
Therefore it follows that 
\begin{align*}
&
f(f-1)(9f-1)
\frac{d^2 P}{d f^2}
+
9
f(f-1) \frac{d P}{d f}
-
(f-1)(9f-1) \frac{d P}{d f}
+
(9f-1) f \frac{d P}{d f}  \\
=&
3 P-Q=
P
\left(
3-\frac{Q}{P}
\right)
=
P
\left(
3
-
\frac{1}{f}
\right)
=
\frac{3f-1}{f} P,
\end{align*}
which proves the proposition. 
\end{proof}

\begin{theorem}(Cooper \cite[pp. 391]{Cooper}) 
\label{thm:P-f-linear-E_2(q)-Heun}
{\it
For every $\tau\in\mathbb{H}^2$ set 
$
\displaystyle 
\tilde{f}(\tau)=
\frac
{
\eta^8(\tau) \eta^4(6\tau)
}
{
\eta^4(2\tau) \eta^8(3\tau)
}.
$ 
Then the inverse function $\tau=\tau(\tilde{f})$ yields $P(\tilde{f})=P(\tau(\tilde{f}))$ and 
$P(\tilde{f})$ satisfies Heun's differential equation:
\begin{equation}
\frac{d^2 P}{d \tilde{f}^2}
+
\left\{
\frac{1}{ \tilde{f} }
+
\frac{1}{ \tilde{f}-1}
+
\frac{1}{\tilde{f}-9}
\right\}
\frac{dP}{d  \tilde{f} }
+
\frac
{
\tilde{f}-3
}
{
\tilde{f} (\tilde{f}-1)(\tilde{f}- 9)
}
P=0. 
\end{equation}
}
\end{theorem}

\begin{proof}
The theorem can be proved by 
changing $\displaystyle f\longrightarrow \frac{1}{f}$ in Proposition \ref{prop:P-f-linear-E_2(q)}. 
\end{proof}

\begin{corollary}
\label{coro:P-u-linear-E_2(q)-Heun}
{\it
For every $\tau\in\mathbb{H}^2$ set 
$
\displaystyle 
u=
\frac{\tilde{f}(\tau)-1}{-8}=
\frac{P-Q}{8 P}. 
$ 
Then the inverse function $\tau=\tau(u)$ yields $P(u)=P(\tau(u))$ and 
$P(u)$ satisfies Heun's differential equation:
\begin{equation}
\frac{d^2 P}{d u^2}
+
\left\{
\frac{1}{ u}
+
\frac{1}{ u+1}
+
\frac{1}{u-\frac18 }
\right\}
\frac{dP}{d  u }
+
\frac
{
u+\frac14
}
{
u (u+1)(u-\frac18)
}
P=0. 
\end{equation}
}
\end{corollary}

\subsection{Heun's equation for $Q(\tau)$}

\begin{theorem}(Cooper \cite[pp. 391]{Cooper})
\label{thm:Q-f-linear-E_2(q)}
{\it
For every $\tau\in\mathbb{H}^2$ set 
$
\displaystyle 
f(\tau)=
\frac
{
\eta^4(2\tau) \eta^8(3\tau)
}
{
\eta^8(\tau) \eta^4(6\tau)
}.
$ 
Then the inverse function $\tau=\tau(f)$ yields $Q(f)=Q(\tau(f))$ and 
$Q(f)$ satisfies Heun's differential equation:
\begin{equation}
\frac{d^2 Q}{d f^2}
+
\left\{
\frac{1}{f}
+
\frac{1}{f-1}
+
\frac{1}{f-\frac19}
\right\}
\frac{dQ}{df}
+
\frac
{
f-\frac13
}
{
f(f-1)(f-\frac19)
}
Q=0. 
\end{equation}
}
\end{theorem}

\begin{proof}
We have 
\begin{equation*}
\frac{d Q}{d f}
=
\frac{d Q}{d \tau}
\frac{d \tau}{d f}
=\frac
{ D Q}{ D f}
=
\frac{Q \left(-27 P^2 Q+24 P Q^2-Q^3+4 Q R\right)}{6 P \left(9 P^2-10 P Q+Q^2\right)}, 
\end{equation*}
and 
\begin{equation*}
\frac{d^2 Q}{d f^2}
=
\frac
{d}{d \tau}
\left(
\frac{d Q}{d f}
\right)
\frac{d \tau}{d f}
=
\frac
{
Q^3 
\left(
243 P^4-486 P^3 Q+300 P^2 Q^2-108 P^2 R-26 P Q^3+80 P Q R+Q^4-4 Q^2 R
\right)
}
{6 P^2 \left(9 P^2-10 P Q+Q^2\right)^2}, 
\end{equation*}
which imply that 
\begin{align*}
&
f(f-1)(9f-1)
\frac{d^2 Q}{d f^2}
+
9
f(f-1)
\frac{d Q}{d f}
+
(f-1)(9f-1)
\frac{d Q}{d f}
+
(9f-1) f
\frac{d Q}{d f}  \\
=&
-9 P + 3 Q
=
-3 Q
\left(
3
\frac{P}{Q}
-1
\right)
=
-3 Q
(3 f-1),
\end{align*}
which proves the theorem. 
\end{proof}

\begin{corollary}
\label{coro:Q-f-linear-E_2(q)}
{\it
For every $\tau\in\mathbb{H}^2$ set 
$
\displaystyle 
v(\tau)=
\frac{f-1}{8}=
\frac
{
P-Q
}
{
8 Q
}.
$ 
Then the inverse function $\tau=\tau(v)$ yields $Q(v)=Q(\tau(v))$ and 
$Q(v)$ satisfies Heun's differential equation:
\begin{equation}
\frac{d^2 Q}{d v^2}
+
\left\{
\frac{1}{v}
+
\frac{1}{v+\frac18}
+
\frac{1}{v+\frac19}
\right\}
\frac{dQ}{dv}
+
\frac
{
v+\frac{1}{12}
}
{
v(v+\frac18)(v+\frac19)
}
Q=0. 
\end{equation}
}
\end{corollary}

\subsection{On $a(\tau)$ and $a(2\tau)$} 
In this subsection we set $l=a(\tau)$ and $m=a(2\tau).$

\begin{theorem}
\label{thm:a-q-f-linear}
{\it
For every $\tau\in\mathbb{H}^2$ set 
$
\displaystyle 
f(\tau)=
\frac
{
\eta^4(2\tau) \eta^8(3\tau)
}
{
\eta^8(\tau) \eta^4(6\tau)
}.
$ 
Then the inverse function $\tau=\tau(f)$ yields $l(f)=a(\tau(f))$ and 
$l(f)$ satisfies the following differential equation:
\begin{equation*}
\frac{d^2 l}{d f^2}
+
\left\{
\frac{1}{f}
+
\frac{1}{f-1}
-
\frac{2}{f-\frac13}
+
\frac{1}{f-\frac19}
\right\}
\frac{dl}{df}
-
\frac
{
8
}
{
27
(f-1)(f-\frac19)
(f-\frac13)^2
}
l=0. 
\end{equation*}
}
\end{theorem}

\begin{proof}
From Matsuda \cite{Matsuda4}, we recall 
\begin{equation*}
a(\tau)=\frac32 P-\frac12 Q \,\,
\mathrm{and} \,\,
c^3(\tau)=\frac{27}{8}(P^3-P^2 Q),
\end{equation*}
which imply that 
\begin{equation*}
\frac{ c^3(\tau) }{ a^3(\tau) }=
\frac{27 (P^3-P^2 Q)}{(3P-Q)^3}=
\frac{27 f^2  (f-1) }{(3 f-1)^3}. 
\end{equation*}
The theorem can be proved by setting 
$
\displaystyle
x=\frac{27 f^2  (f-1) }{(3 f-1)^3}
$ 
in Theorem \ref{thm:a-x-hypergeometric-E_2(q)}. 
\end{proof}

\begin{theorem}
\label{thm:a-q^2-f-linear}
{\it
For every $\tau\in\mathbb{H}^2$ set 
$
\displaystyle 
f(\tau)=
\frac
{
\eta^4(2\tau) \eta^8(3\tau)
}
{
\eta^8(\tau) \eta^4(6\tau)
}.
$ 
Then the inverse function $\tau=\tau(f)$ yields $m(f)=a(2\tau(f))$ and 
$m(f)$ satisfies the following differential equation:
\begin{equation*}
\frac{d^2 m }{d f^2}
+
\left\{
\frac{1}{f}
+
\frac{1}{f-1}
-
\frac{2}{f+\frac13}
+
\frac{1}{f-\frac19}
\right\}
\frac{d m}{df}
-
\frac
{
2
}
{
3
f
(f+\frac13)^2
}
m=0. 
\end{equation*}
}
\end{theorem}

\begin{proof}
From Matsuda \cite{Matsuda4}, we recall 
\begin{equation*}
a(2\tau)=\frac34 P+\frac14 Q \,\,
\mathrm{and} \,\,
c^3(2\tau)=\frac{27}{64}(P^3-2P^2 Q+P Q^2),
\end{equation*}
which imply that 
\begin{equation*}
\frac{ c^3(2\tau) }{ a^3(2\tau) }=
\frac{27 P (P-Q)^2  }{(3P+Q)^3}=
\frac{27 (f-1)^2 f}{(3 f+1)^3}. 
\end{equation*}
The theorem can be proved by setting 
$
\displaystyle
x=\frac{27 (f-1)^2 f}{(3 f+1)^3} 
$ 
in Theorem \ref{thm:a-x-hypergeometric-E_2(q)}. 
\end{proof}

\subsection{On $b(\tau)$ and $b(2\tau)$} 
In this subsection we set $l=b(\tau)$ and $m=b(2\tau).$

\begin{theorem}
\label{thm:b-q-f-linear}
{\it
For every $\tau\in\mathbb{H}^2$ set 
$
\displaystyle 
f(\tau)=
\frac
{
\eta^4(2\tau) \eta^8(3\tau)
}
{
\eta^8(\tau) \eta^4(6\tau)
}.
$ 
Then the inverse function $\tau=\tau(f)$ yields $l(f)=b(\tau(f))$ and 
$l(f)$ satisfies the following differential equation:
\begin{equation*}
\frac{d^2 l}{d f^2}
+
\left\{
\frac{1}{f}
+
\frac{1}{f-1}
+
\frac{ \frac13 }{f-\frac19}
\right\}
\frac{dl}{df}
+
\frac
{
4
(f-\frac13)
}
{
9
(f-1)(f-\frac19)^2
}
l=0. 
\end{equation*}
}
\end{theorem}

\begin{proof}
From Matsuda \cite{Matsuda4}, we recall 
\begin{equation*}
a(\tau)=\frac32 P-\frac12 Q \,\,
\mathrm{and} \,\,
b^3(\tau)=\frac{1}{8} (9P Q^2-Q^3),
\end{equation*}
which imply that 
\begin{equation*}
\frac{ a^3(\tau) }{ b^3(\tau) }=
\frac{(3 P-Q)^3}{Q^2 (9 P-Q)}=
\frac{(3 f-1)^3}{9 f-1}. 
\end{equation*}
The theorem can be proved by setting 
$
\displaystyle
y=\frac{(3 f-1)^3}{9 f-1}
$ 
in Theorem \ref{thm:b-y-hypergeometric-E_2(q)}. 
\end{proof}

\begin{theorem}
\label{thm:b-q^2-f-linear}
{\it
For every $\tau\in\mathbb{H}^2$ set 
$
\displaystyle 
f(\tau)=
\frac
{
\eta^4(2\tau) \eta^8(3\tau)
}
{
\eta^8(\tau) \eta^4(6\tau)
}.
$ 
Then the inverse function $\tau=\tau(f)$ yields $m(f)=b(2\tau(f))$ and 
$m(f)$ satisfies the following differential equation:
\begin{equation*}
\frac{d^2 m }{d f^2}
+
\left\{
\frac{1}{f}
+
\frac{1}{f-1}
+
\frac{\frac13}{f-\frac19}
\right\}
\frac{d m}{df}
+
\frac
{
f+\frac13
}
{
9
f
(f-\frac19)^2
}
m=0. 
\end{equation*}
}
\end{theorem}

\begin{proof}
From Matsuda \cite{Matsuda4}, we recall 
\begin{equation*}
a(2\tau)=\frac34 P+\frac14 Q \,\,
\mathrm{and} \,\,
b^3(2\tau)=\frac{1}{64} (81 P^2 Q-18 P Q^2+Q^3),
\end{equation*}
which imply that 
\begin{equation*}
\frac{ a^3(2\tau) }{ b^3(2\tau) }=
\frac{(3 P+Q)^3}{Q (Q-9 P)^2}=
\frac{(3 f+1)^3}{(1-9 f)^2}. 
\end{equation*}
The theorem can be proved by setting 
$
\displaystyle
y=\frac{(3 f+1)^3}{(1-9 f)^2}
$ 
in Theorem \ref{thm:b-y-hypergeometric-E_2(q)}. 
\end{proof}

\subsection{On $c(\tau)$ and $c(2\tau)$} 
In this subsection we set $l=c(\tau)$ and $m=c(2\tau).$

\begin{theorem}
\label{thm:c-q-f-linear}
{\it
For every $\tau\in\mathbb{H}^2$ set 
$
\displaystyle 
f(\tau)=
\frac
{
\eta^4(2\tau) \eta^8(3\tau)
}
{
\eta^8(\tau) \eta^4(6\tau)
}.
$ 
Then the inverse function $\tau=\tau(f)$ yields $l(f)=c(\tau(f))$ and 
$l(f)$ satisfies the following differential equation:
\begin{equation*}
\frac{d^2 l}{d f^2}
+
\left\{
-
\frac{\frac13 }{f}
+
\frac{\frac13}{f-1}
+
\frac{ 1 }{f-\frac19}
\right\}
\frac{dl}{df}
+
\frac
{
4
(f-\frac13)
}
{
27 f^2(f-1)^2(f-\frac19) 
}
l=0. 
\end{equation*}
}
\end{theorem}

\begin{proof}
From Matsuda \cite{Matsuda4}, we recall 
\begin{equation*}
a(\tau)=\frac32 P-\frac12 Q \,\,
\mathrm{and} \,\,
c^3(\tau)=\frac{27}{8}(P^3-P^2 Q),
\end{equation*}
which imply that 
\begin{equation*}
\frac{ a^3(\tau) }{ c^3(\tau) }=
\frac{(3P-Q)^3}{27 (P^3-P^2 Q)}=
\frac{(3 f-1)^3}{27 f^2  (f-1) }. 
\end{equation*}
The theorem can be proved by setting 
$
\displaystyle
z=\frac{(3 f-1)^3}{27 f^2  (f-1) }
$ 
in Theorem \ref{thm:c-z-hypergeometric-E_2(q^3)}. 
\end{proof}

\begin{theorem}
\label{thm:c-q^2-f-linear}
{\it
For every $\tau\in\mathbb{H}^2$ set 
$
\displaystyle 
f(\tau)=
\frac
{
\eta^4(2\tau) \eta^8(3\tau)
}
{
\eta^8(\tau) \eta^4(6\tau)
}.
$ 
Then the inverse function $\tau=\tau(f)$ yields $m(f)=c(2\tau(f))$ and 
$m(f)$ satisfies the following differential equation:
\begin{equation*}
\frac{d^2 m}{d f^2}
+
\left\{
\frac{\frac13 }{f}
-
\frac{\frac13}{f-1}
+
\frac{ 1 }{f-\frac19}
\right\}
\frac{d m }{df}
+
\frac
{
(f+\frac13)
}
{
3 f^2(f-1)^2   
}
m=0. 
\end{equation*}
}
\end{theorem}

\begin{proof}
From Matsuda \cite{Matsuda4}, we recall 
\begin{equation*}
a(2\tau)=\frac34 P+\frac14 Q \,\,
\mathrm{and} \,\,
c^3(2\tau)=\frac{27}{64}(P^3-2P^2 Q+P Q^2),
\end{equation*}
which imply that 
\begin{equation*}
\frac{ a^3(2\tau) }{ c^3(2\tau) }  =
\frac{(3P+Q)^3}{27 P (P-Q)^2  }=
\frac{(3 f+1)^3}{27 (f-1)^2 f}
\end{equation*}
The theorem can be proved by setting 
$
\displaystyle
z=\frac{(3 f+1)^3}{27 (f-1)^2 f}
$ 
in Theorem \ref{thm:c-z-hypergeometric-E_2(q^3)}. 
\end{proof}

\section{Heun's equations for modular forms of level 6 (2)}
\label{sec:Heun-g-(2/3,1)}

Throughout this section we set 
\begin{align*}
P(\tau)=&
\frac
{
\eta^6(2\tau) \eta(3\tau)
}
{
\eta^3(\tau) \eta^2(6\tau)
}
=1+3\sum_{n=1}^{\infty} (d_{1,6}(n)-d_{5,6}(n)) q^n,   \\
Q(\tau)
=
&
\frac
{
\eta(\tau) \eta^6(6\tau)
}
{
\eta^2(2\tau) \eta^3(3\tau)
}
=
\sum_{n=1}^{\infty} (d_{1,6}(n)-d_{5,6}(n)) q^n-2 \sum_{n=1}^{\infty} ( d_{1,3}(n)-d_{2,3}(n) )q^{2n},  \\
R(\tau)=&E_2(6\tau), 
\end{align*}
and 
\begin{equation*}
g(\tau)=\frac{Q(\tau)}{P(\tau)}
=
\frac
{
\eta^4(\tau) \eta^8(6\tau)
}
{
\eta^8(2\tau) \eta^4(3\tau)
}.
\end{equation*}

In \cite{Matsuda4}, we proved that 
\begin{align*}
DP=&\frac{-P^3+12P^2Q-15PQ^2+PR}{2},    \,\,
DQ=\frac{P^2Q-8PQ^2+3Q^3+QR}{2} \\
DR=&\frac{-P^4+12P^3Q-30P^2Q^2+12PQ^3-9Q^4+R^2}{2},  \,\,\,D=\frac{1}{2 \pi i} \frac{d}{d\tau}.   
\end{align*}

\subsection{On the modular function of level six }

\begin{theorem}
\label{thm:ODE-g-(2/3,1)}
{\it
For every $\tau\in\mathbb{H}^2$ set 
\begin{equation*}
g(\tau)=
\frac
{
\eta^4(\tau) \eta^8(6\tau)
}
{
\eta^8(2 \tau) \eta^4(3\tau)
}.
\end{equation*}
Then we have 
\begin{equation*}
\{g,\tau\}
+
\frac{81 g^4-108 g^3+102 g^2-12 g+1}{2 (9 g-1)^2 (g-1)^2 g^2}
(g^{\prime})^2=0.
\end{equation*}
}
\end{theorem}

\begin{proof}
From Matsuda \cite{Matsuda4} recall the following formula,
\begin{equation*}
j(\tau)=
\frac
{
(3g+1)^3 (243 g^3-405 g^2+225 g+1)^3
}
{
g(g-1)^2(9g-1)^6
}.
\end{equation*}
The theorem can be proved in the same way as Theorem \ref{thm:ODE-f-(1,2/3)}. 
\end{proof}

\subsection{Heun's equation for the modular form of level six  }

\begin{theorem}(Cooper \cite[pp. 391]{Cooper})
\label{thm:P-g-linear-E_2(q^6)}
{\it
For every $\tau\in\mathbb{H}^2$ set 
$
\displaystyle 
g(\tau)=
\frac
{
\eta^4(\tau) \eta^8(6\tau)
}
{
\eta^8(2\tau) \eta^4(3\tau)
}.
$ 
Then the inverse function $\tau=\tau(g)$ yields $P(g)=P(\tau(g))$ and 
$P(g)$ satisfies Heun's differential equation:
\begin{equation}
\frac{d^2 P}{d g^2}
+
\left\{
\frac{1}{g}
+
\frac{1}{g-1}
+
\frac{1}{g-\frac19}
\right\}
\frac{dP}{dg}
+
\frac
{
g-\frac13
}
{
g(g-1)(g-\frac19)
}
P=0. 
\end{equation}
}
\end{theorem}

\begin{proof}
We first have 
\begin{equation*}
Dg=
D
\left(
\frac{Q}{P}
\right)
=
\frac{Q (P-Q) (P-9 Q)}{P}.
\end{equation*}
We next obtain 
\begin{equation*}
\frac{d P}{d g}
=
\frac{d P}{d \tau}
\frac{d \tau}{d g}
=
\frac{D P}{D g}
=
\frac{P \left(-P^3+12 P^2 Q-15 P Q^2+P R\right)}{2 Q (P-Q) (P-9 Q)}, 
\end{equation*}
and 
\begin{equation*}
\frac{d^2 P}{d g^2}
=
\frac{d}{d \tau}
\left(
\frac{d P}{d g}
\right)
\frac{d\tau}{d g}
=
\frac{P^3 \left(P^4-26 P^3 Q+204 P^2 Q^2-P^2 R-390 P Q^3+20 P Q R+243 Q^4-27 Q^2 R\right)}{2 Q^2 (P-9 Q)^2 (P-Q)^2},
\end{equation*}
which imply that 
\begin{align*}
&
g(g-1)(9g-1)
\frac{d^2 P}{d g^2}
+
9
g(g-1)
\frac{d P}{d g}
+
(g-1)(9g-1)
\frac{d P}{d g}
+
(9g-1) g
\frac{d P}{d g}  \\
=&
3(P-3 Q)
=
3P
\left(
1
-
3
\frac{Q}{P}
\right)
=
3P(1-3 g),
\end{align*}
which proves the theorem. 
\end{proof}

\subsection{On $a(\tau)$ and $a(2\tau)$} 
In this subsection we set $l=a(\tau)$ and $m=a(2\tau).$

\begin{theorem}
\label{thm:a-q-g-linear}
{\it
For every $\tau\in\mathbb{H}^2$ set 
$
\displaystyle 
g(\tau)=
\frac
{
\eta^4(\tau) \eta^8(6\tau)
}
{
\eta^8(2\tau) \eta^4(3\tau)
}.
$ 
Then the inverse function $\tau=\tau(g)$ yields $l(g)=a(\tau(g))$ and 
$l(g)$ satisfies the following differential equation:
\begin{equation*}
\frac{d^2 l}{d g^2}
+
\left\{
\frac{1}{g }
+
\frac{1}{g-1}
-
\frac{2}{g+\frac13}
+
\frac{1}{g-\frac19}
\right\}
\frac{dl}{dg }
-
\frac
{
2
}
{
3
g(g+\frac13)^2
}
l=0. 
\end{equation*}
}
\end{theorem}

\begin{proof}
From Matsuda \cite{Matsuda4}, we recall 
\begin{equation*}
a(\tau)=P+3 Q \,\,
\mathrm{and} \,\,
c^3(\tau)=27 P^2 Q-54 P Q^2+27 Q^3,
\end{equation*}
which imply that 
\begin{equation*}
\frac{ c^3(\tau) }{ a^3(\tau) }=
\frac{27 Q (P-Q)^2}{(P+3 Q)^3}=
\frac{27 (1-g)^2 g}{(3 g+1)^3}.  
\end{equation*}
The theorem can be proved by setting 
$
\displaystyle
x=\frac{27 (1-g)^2 g}{(3 g+1)^3}
$ 
in Theorem \ref{thm:a-x-hypergeometric-E_2(q)}. 
\end{proof}

\begin{theorem}
\label{thm:a-q^2-g-linear}
{\it
For every $\tau\in\mathbb{H}^2$ set 
$
\displaystyle 
g(\tau)=
\frac
{
\eta^4(\tau) \eta^8(6\tau)
}
{
\eta^8(2\tau) \eta^4(3\tau)
}.
$ 
Then the inverse function $\tau=\tau(g)$ yields $m(g)=a(2\tau(g))$ and 
$m(g )$ satisfies the following differential equation:
\begin{equation*}
\frac{d^2 m }{d g^2}
+
\left\{
\frac{1}{g }
+
\frac{1}{g-1}
-
\frac{2}{f-\frac13}
+
\frac{1}{g-\frac19}
\right\}
\frac{d m}{dg }
-
\frac
{
8
}
{
27
(g-1)
(g-\frac19)
(g-\frac13)^2
}
m=0. 
\end{equation*}
}
\end{theorem}

\begin{proof}
From Matsuda \cite{Matsuda4}, we recall 
\begin{equation*}
a(2\tau)=P-3 Q \,\,
\mathrm{and} \,\,
c^3(2\tau)=27 P Q^2-27 Q^3,  
\end{equation*}
which imply that 
\begin{equation*}
\frac{ c^3(2\tau) }{ a^3(2\tau) }=
\frac{27 Q^2 (P-Q)}{(P-3 Q)^3}=
\frac{27 (1-g) g^2}{(1-3 g)^3}.  
\end{equation*}
The theorem can be proved by setting 
$
\displaystyle
x=\frac{27 (1-g) g^2}{(1-3 g)^3}
$ 
in Theorem \ref{thm:a-x-hypergeometric-E_2(q)}. 
\end{proof}

\subsection{On $b(\tau)$ and $b(2\tau)$} 
In this subsection we set $l=b(\tau)$ and $m=b(2\tau).$

\begin{theorem}
\label{thm:b-q-g-linear}
{\it
For every $\tau\in\mathbb{H}^2$ set 
$
\displaystyle 
g(\tau)=
\frac
{
\eta^4(\tau) \eta^8(6\tau)
}
{
\eta^8(2\tau) \eta^4(3\tau)
}.
$ 
Then the inverse function $\tau=\tau(g)$ yields $l(g)=b(\tau(g))$ and 
$l(g)$ satisfies the following differential equation:
\begin{equation*}
\frac{d^2 l}{d g^2}
+
\left\{
\frac{1}{f}
+
\frac{1}{f-1}
-
\frac{ \frac13 }{f-\frac19}
\right\}
\frac{dl}{dg}
+
\frac
{
(g+\frac13)
}
{
9
g(g-\frac19)^2
}
l=0. 
\end{equation*}
}
\end{theorem}

\begin{proof}
From Matsuda \cite{Matsuda4}, we recall 
\begin{equation*}
a(\tau)=P+3Q \,\,
\mathrm{and} \,\,
b^3(\tau)=P^3-18 P^2 Q+81 P Q^2,  
\end{equation*}
which imply that 
\begin{equation*}
\frac{ a^3(\tau) }{ b^3(\tau) }=
\frac{(P+3 Q)^3}{P (P-9 Q)^2}=
\frac{(3 g+1)^3}{(1-9 g)^2}.
\end{equation*}
The theorem can be proved by setting 
$
\displaystyle
y=\frac{(3 g+1)^3}{(1-9 g)^2}
$ 
in Theorem \ref{thm:b-y-hypergeometric-E_2(q)}. 
\end{proof}

\begin{theorem}
\label{thm:b-q^2-g-linear}
{\it
For every $\tau\in\mathbb{H}^2$ set 
$
\displaystyle 
g(\tau)=
\frac
{
\eta^4(\tau) \eta^8(6\tau)
}
{
\eta^8(2\tau) \eta^4(3\tau)
}.
$ 
Then the inverse function $\tau=\tau(g)$ yields $m(g)=b(2\tau(g))$ and 
$m(g)$ satisfies the following differential equation:
\begin{equation*}
\frac{d^2 m }{d g^2}
+
\left\{
\frac{1}{g}
+
\frac{1}{g-1}
+
\frac{\frac13}{g-\frac19}
\right\}
\frac{d m}{dg}
+
\frac
{
4
(
g-\frac13
)
}
{
9
(g-1)
(g-\frac19)^2
}
m=0. 
\end{equation*}
}
\end{theorem}

\begin{proof}
From Matsuda \cite{Matsuda4}, we recall 
\begin{equation*}
a(2\tau)=P-3Q,  \,\,
\mathrm{and} \,\,
b^3(2\tau)=P^3-9 P^2 Q, 
\end{equation*}
which imply that 
\begin{equation*}
\frac{ a^3(2\tau) }{ b^3(2\tau) }=
\frac{(P-3 Q)^3}{P^2 (P-9 Q)}=
\frac{(1-3 g)^3}{1-9 g}. 
\end{equation*}
The theorem can be proved by setting 
$
\displaystyle
y=\frac{(1-3 g)^3}{1-9 g}
$ 
in Theorem \ref{thm:b-y-hypergeometric-E_2(q)}. 
\end{proof}

\subsection{On $c(\tau)$ and $c(2\tau)$} 
In this subsection we set $l=c(\tau)$ and $m=c(2\tau).$

\begin{theorem}
\label{thm:c-q-g-linear}
{\it
For every $\tau\in\mathbb{H}^2$ set 
$
\displaystyle 
g(\tau)=
\frac
{
\eta^4(\tau) \eta^8(6\tau)
}
{
\eta^8(2\tau) \eta^4(3\tau)
}.
$ 
Then the inverse function $\tau=\tau(g)$ yields $l(g)=c(\tau(g))$ and 
$l(g)$ satisfies the following differential equation:
\begin{equation*}
\frac{d^2 l}{d g^2}
+
\left\{
\frac{\frac13 }{g}
-
\frac{\frac13}{g-1}
+
\frac{ 1 }{ g-\frac19}
\right\}
\frac{dl}{dg}
+
\frac
{
g+\frac13
}
{
3 g^2(g-1)^2
}
l=0. 
\end{equation*}
}
\end{theorem}

\begin{proof}
From Matsuda \cite{Matsuda4}, we recall 
\begin{equation*}
a(\tau)=P+3 Q \,\,
\mathrm{and} \,\,
c^3(\tau)=27 P^2 Q - 54 P Q^2 + 27 Q^3,     
\end{equation*}
which imply that 
\begin{equation*}
\frac{ a^3(\tau) }{ c^3(\tau) }=
\frac{(P+3 Q)^3}{27 Q (P-Q)^2}=
\frac{(3 g+1)^3}{27 (1-g)^2 g}. 
\end{equation*}
The theorem can be proved by setting 
$
\displaystyle
z=\frac{(3 g+1)^3}{27 (1-g)^2 g}
$ 
in Theorem \ref{thm:c-z-hypergeometric-E_2(q^3)}. 
\end{proof}

\begin{theorem}
\label{thm:c-q^2-g-linear}
{\it
For every $\tau\in\mathbb{H}^2$ set 
$
\displaystyle 
g(\tau)=
\frac
{
\eta^4(\tau) \eta^8(6\tau)
}
{
\eta^8(2\tau) \eta^4(3\tau)
}.
$ 
Then the inverse function $\tau=\tau(g)$ yields $m(g)=c(2\tau(g))$ and 
$m(g)$ satisfies the following differential equation:
\begin{equation*}
\frac{d^2 m}{d g^2}
+
\left\{
-
\frac{\frac13 }{g}
+
\frac{\frac13}{g-1}
+
\frac{ 1 }{g-\frac19}
\right\}
\frac{d m }{dg}
+
\frac
{
4
(g-\frac13)
}
{
27 g^2(g-1)^2
(g-\frac19)   
}
m=0. 
\end{equation*}
}
\end{theorem}

\begin{proof}
From Matsuda \cite{Matsuda4}, we recall 
\begin{equation*}
a(2\tau)=P-3Q \,\,
\mathrm{and} \,\,
c^3(2\tau)=27 P Q^2-27 Q^3,   
\end{equation*}
which imply that 
\begin{equation*}
\frac{ a^3(2\tau) }{ c^3(2\tau) }  =
\frac{(P-3 Q)^3}{27 Q^2 (P-Q)}=
\frac{(1-3 g)^3}{27 (1-g) g^2}.  
\end{equation*}
The theorem can be proved by setting 
$
\displaystyle
z=\frac{(1-3 g)^3}{27 (1-g) g^2}
$ 
in Theorem \ref{thm:c-z-hypergeometric-E_2(q^3)}. 
\end{proof}

\section{Heun's equations for modular forms of level 6 (3)}
\label{sec:Heun-t-(0,1/3)}

Throughout this section we set 
\begin{align*}
P(\tau)=&a(\tau)=\sum_{m,n\in\mathbb{Z}} q^{m^2+mn+n^2},  \,\,\
Q(\tau)=
\frac
{
\eta\left(\tau/2 \right) \eta^6(3\tau)
}
{
\eta^2(\tau) \eta^3 \left(3\tau/2 \right)
}
=\sum_{n=1}^{\infty}(d_{1,3}^{*}(n)-d_{2,3}^{*}(n)) q^{\frac{n}{2}},  \\
T(\tau)
=
&
\frac
{
\eta^6(\tau) \eta(3\tau/2)
}
{
\eta^3(\tau/2) \eta^2(3\tau)
}
=
1+3\sum_{n=1}^{\infty} (d_{1,3}^{*}(n)-d_{2,3}^{*}(n)) q^{\frac{n}{2}}+6 \sum_{n=1}^{\infty} ( d_{1,3}(n)-d_{2,3}(n) )q^n,  \\
S(\tau)=&E_2(\tau), 
\end{align*}
and
\begin{equation*}
t=t(\tau)=\frac{Q(\tau)}{T(\tau)}=\frac
{
\eta^4 \left( \tau/2 \right) \eta^8(3\tau)
}
{
\eta^8( \tau) \eta^4 \left( 3\tau/2 \right)
}.
\end{equation*}

In \cite{Matsuda4}, we proved that 
\begin{align*}
D
Q=&
\frac{ 27 Q^3-36 Q^2 T+5 Q T^2 +Q S }{12},   \,\,
D
T=
\frac{-27 Q^2 T+24 Q T^2-T^3 +S T }{12},   \\
D
S=&
\frac{-729 Q^4+972 Q^3 T-270 Q^2 T^2+12 Q T^3-T^4  +S^2}{12},  \,\,
D=\frac{1}{2 \pi i} \frac{d}{d\tau}. 
\end{align*}

\subsection{On the modular function of level six}

\begin{theorem}
\label{thm:ODE-t-(0,1/3)}
{\it
For every $\tau\in\mathbb{H}^2$ set 
\begin{equation*}
t(\tau)=
\frac
{
\eta^4 \left( \tau/2 \right) \eta^8(3\tau)
}
{
\eta^8( \tau) \eta^4 \left( 3\tau/2 \right)
}.
\end{equation*}
Then we have 
\begin{equation*}
\{t,\tau\}
+
\frac{81 t^4-108 t^3+102 t^2-12 t+1}{2 t^2 (9 t-1)^2 (t-1)^2 }
(t^{\prime})^2=0.
\end{equation*}
}
\end{theorem}

\begin{proof}
From Matsuda \cite{Matsuda4} recall the following formula,
\begin{equation*}
j(\tau)=
\frac{(3 t-1)^3 \left(243 t^3-243 t^2+9 t-1\right)^3}{ t^2 (t-1)  (9 t-1)^3}. 
\end{equation*}
The theorem can be proved in the same way as Theorem \ref{thm:ODE-f-(1,2/3)}. 
\end{proof}

\subsection{Heun's equation for the modular form of level six  }

\begin{theorem}
\label{thm:T-t-linear-E_2(q)}
{\it
For every $\tau\in\mathbb{H}^2$ set 
$
\displaystyle 
t(\tau)=
\frac
{
\eta^4 \left( \tau/2 \right) \eta^8(3\tau)
}
{
\eta^8( \tau) \eta^4 \left( 3\tau/2 \right)
}.
$ 
Then the inverse function $\tau=\tau(t)$ yields $T(t)=T(\tau(t))$ and 
$T(t)$ satisfies Heun's differential equation:
\begin{equation}
\frac{d^2 T}{d t^2}
+
\left\{
\frac{1}{t}
+
\frac{1}{t-1}
+
\frac{1}{t-\frac19}
\right\}
\frac{dT}{dt}
+
\frac
{
t-\frac13
}
{
t(t-1)(t-\frac19)
}
T=0. 
\end{equation}
}
\end{theorem}

\begin{proof}
We first have 
\begin{equation*}
D t=D
\left(
\frac{Q}{T}
\right)
=\frac{Q (Q-T) (9 Q-T)}{2 T}. 
\end{equation*}
We next obtain  
\begin{equation*}
\frac{d T}{d t}
=
\frac{d T}{d \tau}
\frac{d \tau}{d t}
=
\frac
{
D T
}
{
D t
}=
-\frac{T^2 \left(27 Q^2-24 Q T-S+T^2\right)}{6 Q (Q-T) (9 Q-T)}
\end{equation*}
and 
\begin{equation*}
\frac{d^2 T}{d t^2}
=
\frac{d}{d\tau}
\left(
\frac{d T}{d t}
\right)
\frac{d \tau}{d t} 
=
\frac{T^3 \left(243 Q^4-486 Q^3 T-27 Q^2 S+300 Q^2 T^2+20 Q S T-26 Q T^3-S T^2+T^4\right)}{6 Q^2 (Q-T)^2 (9 Q-T)^2},
\end{equation*}
which imply that 
\begin{align*}
&
t(t-1)(9t-1) \frac{d^2 T}{d t^2}
+
9t(t-1)\frac{d T}{d t}
+
(t-1)(9t-1)\frac{d T}{d t}
+
t(9t-1) \frac{d T}{d t}   \\
=&
-9 Q+3 T
=
3 T
\left(
-3 
\frac{Q}{T}
+1
\right)
=3T(-3 t+1),
\end{align*}
which proves the theorem. 
\end{proof}

\subsection{On $a(\tau/2)$ and $a(\tau)$} 
In this subsection we set $k=a(\tau/2)$ and $l=a(\tau).$

\begin{theorem}
\label{thm:a-q^{1/2}-t-linear}
{\it
For every $\tau\in\mathbb{H}^2$ set 
$
\displaystyle 
t(\tau)=
\frac
{
\eta^4 \left( \tau/2 \right) \eta^8(3\tau)
}
{
\eta^8( \tau) \eta^4 \left( 3\tau/2 \right)
}.
$ 
Then the inverse function $\tau=\tau(t)$ yields $k(t)=a(\tau(t)/2)$ and 
$k(t )$ satisfies the following differential equation:
\begin{equation*}
\frac{d^2 k }{d t^2}
+
\left\{
\frac{1}{t }
+
\frac{1}{t-1}
-
\frac{2}{t+\frac13}
+
\frac{1}{t-\frac19}
\right\}
\frac{d k}{d t }
-
\frac
{
2
}
{
3
t
(t+\frac13)^2
}
k=0. 
\end{equation*}
}
\end{theorem}

\begin{proof}
From Matsuda \cite{Matsuda4}, we recall 
\begin{equation*}
a(\tau/2)=P+6 Q=T+3 Q \,\,
\mathrm{and} \,\,
c^3(\tau/2)=27 P^2 Q+108 P Q^2+108 Q^3=27 Q^3-54 Q^2T+27 Q T^2,    
\end{equation*}
which imply that 
\begin{equation*}
\frac{ c^3(2\tau) }{ a^3(2\tau) }=
\frac{27 Q (Q-T)^2}{(3 Q+T)^3}=
\frac{27 (t-1)^2 t}{(3 t+1)^3}.    
\end{equation*}
The theorem can be proved by setting 
$
\displaystyle
x=\frac{27 (t-1)^2 t}{(3 t+1)^3}
$ 
in Theorem \ref{thm:a-x-hypergeometric-E_2(q)}. 
\end{proof}

\begin{theorem}
\label{thm:a-q-t-linear}
{\it
For every $\tau\in\mathbb{H}^2$ set 
$
\displaystyle 
t(\tau)=
\frac
{
\eta^4 \left( \tau/2 \right) \eta^8(3\tau)
}
{
\eta^8( \tau) \eta^4 \left( 3\tau/2 \right)
}.
$ 
Then the inverse function $\tau=\tau(t)$ yields $l(t)=a(\tau(t))$ and 
$l(t)$ satisfies the following differential equation:
\begin{equation*}
\frac{d^2 l}{d g^2}
+
\left\{
\frac{1}{g }
+
\frac{1}{g-1}
-
\frac{2}{g+\frac13}
+
\frac{1}{g-\frac19}
\right\}
\frac{dl}{dt }
-
\frac
{
8
}
{
27
(t-1)(t-\frac19) (t-\frac13)^2
}
l=0. 
\end{equation*}
}
\end{theorem}

\begin{proof}
From Matsuda \cite{Matsuda4}, we recall 
\begin{equation*}
P=a(\tau)=T-3 Q \,\,
\mathrm{and} \,\,
c^3(\tau)=27P Q^2+54 Q^3=-27Q^3+27 Q^2 T, 
\end{equation*}
which imply that 
\begin{equation*}
\frac{ c^3(\tau) }{ a^3(\tau) }=
\frac{27 Q^2 (Q-T)}{(3 Q-T)^3}=
\frac{27 (t-1) t^2}{(3 t-1)^3}.  
\end{equation*}
The theorem can be proved by setting 
$
\displaystyle
x=\frac{27 (t-1) t^2}{(3 t-1)^3}
$ 
in Theorem \ref{thm:a-x-hypergeometric-E_2(q)}. 
\end{proof}

\subsection{On $b(\tau/2)$ and $b(\tau)$} 
In this subsection we set $k=b(\tau/2)$ and $l=b(\tau).$

\begin{theorem}
\label{thm:b-q^{1/2}-t-linear}
{\it
For every $\tau\in\mathbb{H}^2$ set 
$
\displaystyle 
t(\tau)=
\frac
{
\eta^4 \left( \tau/2 \right) \eta^8(3\tau)
}
{
\eta^8( \tau) \eta^4 \left( 3\tau/2 \right)
}.
$ 
Then the inverse function $\tau=\tau(t)$ yields $k(t)=b(\tau(t)/2)$ and 
$k(t)$ satisfies the following differential equation:
\begin{equation*}
\frac{d^2 k}{d t^2}
+
\left\{
\frac{1}{t}
+
\frac{1}{t-1}
-
\frac{ \frac13 }{t-\frac19}
\right\}
\frac{dk}{dt}
+
\frac
{
(t+\frac13)
}
{
9
t(t-\frac19)^2
}
k=0. 
\end{equation*}
}
\end{theorem}

\begin{proof}
From Matsuda \cite{Matsuda4}, we recall 
\begin{equation*}
a(\tau/2)=P+6Q=T+3Q \,\,
\mathrm{and} \,\,
b^3(\tau/2)=P^3-9 P^2 Q+108 Q^3=81 Q^2 T-18 Q T^2+T^3, 
\end{equation*}
which imply that 
\begin{equation*}
\frac{ a^3(\tau/2) }{ b^3(\tau/2) }=
\frac{(3 Q+T)^3}{T (9 Q-T)^2}=
\frac{(3 t+1)^3}{(9 t-1)^2}. 
\end{equation*}
The theorem can be proved by setting 
$
\displaystyle
y=\frac{(3 t+1)^3}{(9 t-1)^2}
$ 
in Theorem \ref{thm:b-y-hypergeometric-E_2(q)}. 
\end{proof}

\begin{theorem}
\label{thm:b-q-t-linear}
{\it
For every $\tau\in\mathbb{H}^2$ set 
$
\displaystyle 
t(\tau)=
\frac
{
\eta^4 \left( \tau/2 \right) \eta^8(3\tau)
}
{
\eta^8( \tau) \eta^4 \left( 3\tau/2 \right)
}.
$ 
Then the inverse function $\tau=\tau(t)$ yields $l(t)=b(\tau(t))$ and 
$l(t)$ satisfies the following differential equation:
\begin{equation*}
\frac{d^2 l }{d t^2}
+
\left\{
\frac{1}{t}
+
\frac{1}{t-1}
+
\frac{\frac13}{t-\frac19}
\right\}
\frac{d l }{d  t}
+
\frac
{
4
(
t-\frac13
)
}
{
9
(t-1)
(t-\frac19)^2
}
l=0. 
\end{equation*}
}
\end{theorem}

\begin{proof}
From Matsuda \cite{Matsuda4}, we recall 
\begin{equation*}
P=a(\tau)=T-3Q,  \,\,
\mathrm{and} \,\,
b^3(\tau)=P^3-27 P Q^2 - 54 Q^3 =T^3-9 Q T^2,
\end{equation*}
which imply that 
\begin{equation*}
\frac{ a^3(\tau) }{ b^3(\tau) }=
\frac{(3 Q-T)^3}{T^2 (9 Q-T)}=
\frac{(3 t-1)^3}{9 t-1}.  
\end{equation*}
The theorem can be proved by setting 
$
\displaystyle
y=\frac{(3 t-1)^3}{9 t-1}  
$ 
in Theorem \ref{thm:b-y-hypergeometric-E_2(q)}. 
\end{proof}

\subsection{On $c(\tau/2)$ and $c(\tau)$} 
In this subsection we set $k=c(\tau/2)$ and $l=c(\tau).$

\begin{theorem}
\label{thm:c-q^{1/2}-t-linear}
{\it
For every $\tau\in\mathbb{H}^2$ set 
$
\displaystyle 
t(\tau)=
\frac
{
\eta^4 \left( \tau/2 \right) \eta^8(3\tau)
}
{
\eta^8( \tau) \eta^4 \left( 3\tau/2 \right)
}.
$ 
Then the inverse function $\tau=\tau(t)$ yields $k(t)=c(\tau(t)/2)$ and 
$k(t)$ satisfies the following differential equation:
\begin{equation*}
\frac{d^2 k  }{d t^2}
+
\left\{
\frac{\frac13 }{t}
-
\frac{\frac13}{t-1}
+
\frac{ 1 }{ t-\frac19}
\right\}
\frac{dk }{d  t}
+
\frac
{
t+\frac13
}
{
3 t^2(t-1)^2
}
k=0. 
\end{equation*}
}
\end{theorem}

\begin{proof}
From Matsuda \cite{Matsuda4}, we recall 
\begin{equation*}
a(\tau/2)=P+6 Q=T+3Q \,\,
\mathrm{and} \,\,
c^3(\tau/2)=27 P^2 Q+108 P Q^2+108 Q^3=27 Q^3-54 Q^2 T+27 QT^2,   
\end{equation*}
which imply that 
\begin{equation*}
\frac{ a^3(\tau/2) }{ c^3(\tau/2) }=
\frac{(3Q+T)^3}{27 Q (Q-T)^2}=
\frac{(3 t+1)^3}{27 (1-t)^2 t}. 
\end{equation*}
The theorem can be proved by setting 
$
\displaystyle
z=\frac{(3 t+1)^3}{27 (1-t)^2 t} 
$ 
in Theorem \ref{thm:c-z-hypergeometric-E_2(q^3)}. 
\end{proof}

\begin{theorem}
\label{thm:c-q-t-linear}
{\it
For every $\tau\in\mathbb{H}^2$ set 
$
\displaystyle 
t(\tau)=
\frac
{
\eta^4 \left( \tau/2 \right) \eta^8(3\tau)
}
{
\eta^8( \tau) \eta^4 \left( 3\tau/2 \right)
}.
$ 
Then the inverse function $\tau=\tau(t)$ yields $l(t)=c(\tau(t))$ and 
$l(t)$ satisfies the following differential equation:
\begin{equation*}
\frac{d^2 l}{d t^2}
+
\left\{
-
\frac{\frac13 }{t}
+
\frac{\frac13}{t-1}
+
\frac{ 1 }{t-\frac19}
\right\}
\frac{d l }{dt}
+
\frac
{
4
(t-\frac13)
}
{
27 t^2(t-1)^2
(t-\frac19)   
}
l=0. 
\end{equation*}
}
\end{theorem}

\begin{proof}
From Matsuda \cite{Matsuda4}, we recall 
\begin{equation*}
P=a(\tau)=T-3Q \,\,
\mathrm{and} \,\,
c^3(\tau)=27 P Q^2+27 Q^3=-27 Q^3+27 Q^2 T,   
\end{equation*}
which imply that 
\begin{equation*}
\frac{ a^3(\tau) }{ c^3(\tau) }  =
\frac{(T-3 Q)^3}{27 Q^2 (T-Q)}=
\frac{(1-3 t)^3}{27 (1-t) t^2}.  
\end{equation*}
The theorem can be proved by setting 
$
\displaystyle
z=\frac{(1-3 t)^3}{27 (1-t) t^2}  
$ 
in Theorem \ref{thm:c-z-hypergeometric-E_2(q^3)}. 
\end{proof}

\section{Heun's equations for modular forms of level 6 (4)}
\label{sec:Heun-u-(0,2/3)}

Throughout this section we set 
\begin{align*}
R(\tau)=&
\frac
{
\eta(\tau)
\eta^3\left(3\tau/2 \right) 
\eta^3(6\tau)
}
{
\eta(\tau/2) \eta \left(2\tau \right)
\eta^3(3\tau)
}
=\sum_{n=1}^{\infty}(d_{1,6}^{*}(n)+ d_{2,6}^{*}(n)   -d_{4,6}^{*}(n) -  d_{5,6}^{*}(n) ) q^{\frac{n}{2}},  \\
U(\tau)
=
&
\frac
{
\eta^3(\tau/2) \eta^3(2\tau)  
\eta(3\tau)
}
{
\eta^3(\tau) \eta(3\tau/2)  \eta(6\tau)
}
=
1-3  \sum_{n=1}^{\infty}(d_{1,6}^{*}(n)+d_{2,6}^{*}(n)  -d_{4,6}^{*}(n)-d_{5,6}^{*}(n))q^{\frac{n}{2}}    \\
&\hspace{60mm}+6  \sum_{n=1}^{\infty}  (d_{1,3}(n)-d_{2,3}(n)) q^{n} \\
S(\tau)=&E_2(\tau), 
\end{align*}
and
\begin{equation*}
u=u(\tau)=\frac{R(\tau)}{U(\tau)}
=
\frac
{
\eta^4 \left( \tau  \right) 
\eta^4(3\tau/2)  
\eta^4(6\tau)
}
{
\eta^4( \tau/2) \eta^4 \left( 2\tau \right) 
\eta^4(3\tau)
}.
\end{equation*}

In \cite{Matsuda4}, we proved that 
\begin{align*}
D
R
=&
\frac{27 R^3+36 R^2 U+5 R U^2  +R S }{12},  \,\,
D
U
= 
\frac{-27 R^2 U-24 R U^2-U^3+S U}{12},  \\
D
S
=&
\frac{-729 R^4-972 R^3 U-270 R^2 U^2-12 R U^3-U^4 +S^2 }{12}, \,\,
D=\frac{1}{2\pi i} \frac{d}{d \tau}. 
\end{align*}

\subsection{On the modular function of level six}

\begin{theorem}
\label{thm:ODE-u-(0,2/3)}
{\it
For every $\tau\in\mathbb{H}^2$ set 
\begin{equation*}
u(\tau)=
\frac
{
\eta^4(\tau)
\eta^4 \left( 3\tau/2 \right) 
\eta^4(6\tau)
}
{
\eta^4 \left( \tau/2 \right)
\eta^4(2\tau)
\eta^4(3\tau)
}.
\end{equation*}
Then we have 
\begin{equation*}
\{u,\tau\}
+
\frac{81 u^4+108 u^3+102 u^2+12 u+1}{2 u^2 (u+1)^2 (9 u+1)^2}
(u^{\prime})^2=0.
\end{equation*}
}
\end{theorem}

\begin{proof}
From Matsuda \cite{Matsuda4} recall the following formula,
\begin{equation*}
j(\tau)=
\frac{(3 u+1)^3 \left(243 u^3+243 u^2+9 u+1\right)^3}{u^2 (u+1) (9 u+1)^3}. 
\end{equation*}
The theorem can be proved in the same way as Theorem \ref{thm:ODE-f-(1,2/3)}. 
\end{proof}

\subsection{Heun's equation for the modular form of level six}

\begin{theorem}
\label{prop:U-u-linear-E_2(q)}
{\it
For every $\tau\in\mathbb{H}^2$ set 
$
\displaystyle 
u(\tau)=
\frac
{
\eta^4(\tau)
\eta^4 \left( 3\tau/2\right) 
\eta^4(6\tau)
}
{
\eta^4 \left(\tau/2 \right)
\eta^4(2\tau)
\eta^4(3\tau)
}.
$ 
Then the inverse function $\tau=\tau(u)$ yields $U(u)=U(\tau(u))$ and 
$U(u)$ satisfies Heun's differential equation:
\begin{equation}
\frac{d^2 U}{d u^2}
+
\left\{
\frac{1}{u}
+
\frac{1}{u+1}
+
\frac{1}{u+\frac19}
\right\}
\frac{dU}{du}
+
\frac
{
u+\frac13
}
{
u(u+1)(u+\frac19)
}
U=0. 
\end{equation}
}
\end{theorem}

\begin{proof}
We first have 
\begin{equation*}
Du=
D
\left(
\frac{R}{U}
\right)
=
\frac{R (R + U) (9 R + U)}{2 U}. 
\end{equation*}
We next obtain  
\begin{equation*}
\frac{d U}{d u}
=
\frac{d U}{d\tau} \frac{d \tau}{d u}
=
\frac{U \left(-27 R^2 U-24 R U^2+S U-U^3\right)}{6 R (R+U) (9 R+U)},
\end{equation*}
and 
\begin{equation*}
\frac{d^2 U}{d u^2}
=
\frac{d}{d \tau}
\left(
\frac{d U}{d u}
\right)
\frac{d \tau}{d u}
=
\frac{U^3 \left(243 R^4+486 R^3 U-27 R^2 S+300 R^2 U^2-20 R S U+26 R U^3-S U^2+U^4\right)}{6 R^2 (R+U)^2 (9 R+U)^2},
\end{equation*}
which imply that 
\begin{align*}
&
u(u+1)(9u+1)\frac{d^2 U}{d u^2}
+
9u(u+1)\frac{d U}{d u}
+
(u+1)(9u+1)\frac{d U}{d u}
+
(9u+1)u\frac{d U}{d u}  \\
=&
-3(3R+U)
=
-3 U
\left(
3
\frac{R}{U}
+1
\right)
=
-3(3u+1)U,
\end{align*}
which proves the theorem. 
\end{proof}

\subsection{On $a(\tau)$} 
In this subsection we set $l=a(\tau).$

\begin{theorem}
\label{thm:a-q-u-linear}
{\it
For every $\tau\in\mathbb{H}^2$ set 
$
\displaystyle 
u(\tau)=
\frac
{
\eta^4(\tau)
\eta^4 \left( 3\tau/2\right) 
\eta^4(6\tau)
}
{
\eta^4 \left(\tau/2 \right)
\eta^4(2\tau)
\eta^4(3\tau)
}.
$ 
Then the inverse function $\tau=\tau(u)$ yields $l(u)=a(\tau(u))$ and 
$l(u )$ satisfies the following differential equation:
\begin{equation*}
\frac{d^2 l }{d u^2}
+
\left\{
\frac{1}{u }
+
\frac{1}{u+1}
-
\frac{2}{u+\frac13}
+
\frac{1}{u+\frac19}
\right\}
\frac{d l}{d u }
-
\frac
{
8
}
{
27
(u+1)
(u+\frac19)
(u+\frac13)^2
}
l=0. 
\end{equation*}
}
\end{theorem}

\begin{proof}
From Matsuda \cite{Matsuda4}, we recall 
\begin{equation*}
P=a(\tau)=U+3 R \,\,
\mathrm{and} \,\,
c^3(\tau)=27 P R^2 - 54 R^3=27 R^3+27 R^2 U,    
\end{equation*}
which imply that 
\begin{equation*}
\frac{ c^3(\tau) }{ a^3(\tau) }=
\frac{27 R^2 (R+U)}{(3 R+U)^3}=
\frac{27 u^2 (u+1)}{(3 u+1)^3}.  
\end{equation*}
The theorem can be proved by setting 
$
\displaystyle
x=\frac{27 u^2 (u+1)}{(3 u+1)^3}
$ 
in Theorem \ref{thm:a-x-hypergeometric-E_2(q)}. 
\end{proof}

\subsection{On $b(\tau)$} 
In this subsection we set $l=b(\tau).$

\begin{theorem}
\label{thm:b-q-u-linear}
{\it
For every $\tau\in\mathbb{H}^2$ set 
$
\displaystyle 
u(\tau)=
\frac
{
\eta^4(\tau)
\eta^4 \left( 3\tau/2\right) 
\eta^4(6\tau)
}
{
\eta^4 \left(\tau/2 \right)
\eta^4(2\tau)
\eta^4(3\tau)
}.
$ 
Then the inverse function $\tau=\tau(u)$ yields $l(u)=b(\tau(u))$ and 
$l(u)$ satisfies the following differential equation:
\begin{equation*}
\frac{d^2 l }{d u^2}
+
\left\{
\frac{1}{u}
+
\frac{1}{u+1}
+
\frac{\frac13}{u+\frac19}
\right\}
\frac{d l }{d  u}
+
\frac
{
4
(
u+\frac13
)
}
{
9
(u+1)
(u+\frac19)^2
}
l=0. 
\end{equation*}
}
\end{theorem}

\begin{proof}
From Matsuda \cite{Matsuda4}, we recall 
\begin{equation*}
P=a(\tau)=U+3 R,  \,\,
\mathrm{and} \,\,
b^3(\tau)=P^3 - 27 P R^2 + 54 R^3 =9 R U^2+U^3, 
\end{equation*}
which imply that 
\begin{equation*}
\frac{ a^3(\tau) }{ b^3(\tau) }=
\frac{(3 R+U)^3}{U^2 (9 R+U)}=
\frac{(3 u+1)^3}{9 u+1}.  
\end{equation*}
The theorem can be proved by setting 
$
\displaystyle
y=\frac{(3 u+1)^3}{9 u+1}
$ 
in Theorem \ref{thm:b-y-hypergeometric-E_2(q)}. 
\end{proof}

\subsection{On $c(\tau)$} 
In this subsection we set $l=c(\tau).$

\begin{theorem}
\label{thm:c-q-u-linear}
{\it
For every $\tau\in\mathbb{H}^2$ set 
$
\displaystyle 
u(\tau)=
\frac
{
\eta^4(\tau)
\eta^4 \left( 3\tau/2\right) 
\eta^4(6\tau)
}
{
\eta^4 \left(\tau/2 \right)
\eta^4(2\tau)
\eta^4(3\tau)
}.
$ 
Then the inverse function $\tau=\tau(u)$ yields $l(u)=c(\tau(u))$ and 
$l(u)$ satisfies the following differential equation:
\begin{equation*}
\frac{d^2 l}{d u^2}
+
\left\{
-
\frac{\frac13 }{u}
+
\frac{\frac13}{u+1}
+
\frac{ 1 }{u+\frac19}
\right\}
\frac{d l }{du}
+
\frac
{
4
(u+\frac13)
}
{
27 u^2(u+1)^2
(u+\frac19)   
}
l=0. 
\end{equation*}
}
\end{theorem}

\begin{proof}
From Matsuda \cite{Matsuda4}, we recall 
\begin{equation*}
P=a(\tau)=U+3R \,\,
\mathrm{and} \,\,
c^3(\tau)=27 P R^2 - 54 R^3=27 R^3+27 R^2 U,    
\end{equation*}
which imply that 
\begin{equation*}
\frac{ a^3(\tau) }{ c^3(\tau) }  =
\frac{(3 R+U)^3}{27 R^2 (R+U)}=
\frac{(3 u+1)^3}{27 u^2 (u+1)}. 
\end{equation*}
The theorem can be proved by setting 
$
\displaystyle
z=\frac{(3 u+1)^3}{27 u^2 (u+1)}
$ 
in Theorem \ref{thm:c-z-hypergeometric-E_2(q^3)}. 
\end{proof}

\section{ODEs for modular forms of level 6 (5)}
\label{sec:Fuchsian-(0,1/3)-(0,2/3)}
In this section we set 
\begin{equation*}
P(\tau)=a(\tau)=\sum_{m,n} q^{m^2+mn+n^2}, \,\,
Q(\tau)=
\frac
{
\eta\left(\tau/2 \right) \eta^6(3\tau)
}
{
\eta^2(\tau) \eta^3 \left(3\tau/2 \right)
}, \,\,
R(\tau)=
\frac
{
\eta(\tau)
\eta^3\left(3\tau/2 \right) 
\eta^3(6\tau)
}
{
\eta(\tau/2) \eta \left(2\tau \right)
\eta^3(3\tau)
}, \,\,
S(\tau)=E_2(\tau),
\end{equation*}
and 
\begin{equation*}
x=\frac{1}{y}=
\frac{R}{Q}=
\frac
{
\eta^3(\tau) \eta^6(3\tau/2) \eta^3(6\tau)
}
{
\eta^2(\tau/2) \eta(2 \tau) \eta^9(3\tau)
}.
\end{equation*}
In \cite{Matsuda4}, we proved that 
\begin{align*}
D P=&\frac{-P^3+108 P Q^2+216 Q^3+P S  }{12   }=\frac{ -P^3+108 P R^2-216 R^3+PS  }{ 12 }, \\
D Q=&\frac{ 5 P^2 Q-6 P Q^2-36Q^3+QS }{ 12 },  \,\,
D R=\frac{ 5 P^2 R+6 P R^2-36 R^3+ RS }{12  },  \\
DS=&\frac{ -P^4-216 P^2 Q^2-432 P Q^3+S^2 }{ 12  }=\frac{ -P^4-216P^2 R^2+432 P R^3+S^2 }{  12}, \,\,D=\frac{1}{2 \pi i} \frac{d}{d\tau},
\end{align*}
and 
\begin{equation}
\label{eqn:(P,Q,R)-relation}
P(Q-R)+2(Q^2-QR+R^2)=0,
\end{equation}
which imply that 
\begin{align*}
DQ
=&
-\frac{Q \left(4 Q^4-8 Q^3 R-48 Q^2 R^2+52 Q R^3-20 R^4-Q^2 S+2 Q R S-R^2 S\right)}{12 (Q-R)^2}, \\
DR
=&
-\frac{R \left(-20 Q^4+52 Q^3 R-48 Q^2 R^2-8 Q R^3+4 R^4-Q^2 S+2 Q R S-R^2 S\right)}{12  (Q-R)^2}, \,\,
\end{align*}
and 
\begin{align*}
DS=&
\frac{-1}{12 (Q-R)^4} \times \\
\times&
\Big(
16 Q^8-64 Q^7 R+1024 Q^6 R^2-2848 Q^5 R^3 +3760 Q^4 R^4 -2848 Q^3 R^5  +1024 Q^2 R^6-64 Q R^7  +16 R^8 \\
&\hspace{5mm}
-Q^4 S^2+4 Q^3 R S^2-6 Q^2 R^2 S^2+4 Q R^3 S^2-R^4 S^2 \Big). 
\end{align*}

\subsection{On the modular function of level 6}

\begin{theorem}
\label{thm:ODE-x-(0,1/3)-(0,2/3)}
{\it
For every $\tau\in\mathbb{H}^2$ set 
$
\displaystyle 
x=
\frac
{
\eta^3(\tau) \eta^6(3\tau/2) \eta^3(6\tau)
}
{
\eta^2(\tau/2) \eta(2 \tau) \eta^9(3\tau)
}.
$ 
Then we have 
\begin{equation*}
\left\{x, \tau \right\}+
\frac{2 \left(x^2-x+1\right)^4}{ x^2 (x-1)^2  (x+1)^2  (x-2)^2(2 x-1)^2  }
(x^{\prime})^2=0. 
\end{equation*}
}
\end{theorem}

\begin{proof}
From \cite{Matsuda4} we recall 
\begin{align*}
j(\tau)
=&
\frac{256 \left(Q^2-Q R+R^2\right)^3 \left(Q^6-3 Q^5 R+60 Q^4 R^2-115 Q^3 R^3+60 Q^2 R^4-3 Q R^5+R^6\right)^3}
{Q^2 R^2 (Q-2 R)^6 (Q-R)^2 (R-2 Q)^6 (Q+R)^6}  \\
=&
\frac{256 \left(x^2-x+1\right)^3 \left(x^6-3 x^5+60 x^4-115 x^3+60 x^2-3 x+1\right)^3}{(1-2 x)^6 (1-x)^2 (x-2)^6 x^2 (x+1)^6}. 
\end{align*}
The theorem can be proved in the same way as Theorem \ref{thm:ODE-f-(1,2/3)}. 
\end{proof}

\subsection{On $Q(\tau)$}

\begin{theorem}
\label{thm:Q-x-linear-(0,1/3)-(0,2/3)}
{\it
For every $\tau\in\mathbb{H}^2$ set 
$
\displaystyle 
x=
\frac
{
\eta^3(\tau) \eta^6(3\tau/2) \eta^3(6\tau)
}
{
\eta^2(\tau/2) \eta(2 \tau) \eta^9(3\tau)
}.
$ 
Then the inverse function $\tau=\tau(x)$ yields $Q(x)=Q(\tau(x))$ and 
$Q(x)$ satisfies the following differential equation:
\begin{equation*}
\frac{d^2 Q}{ d x^2}
+
\left\{
\frac{1}{x}
-
\frac{1}{x-1}
+
\frac{1}{x+1}
+
\frac{1}{x-2}
+
\frac{1}{x-\frac12}
\right\}
\frac{d Q}{d x}
+
\frac{2  \left(x^2-x+1\right)}{(x-1)^2 (x+1) (2 x-1)} Q
=0.
\end{equation*}
}
\end{theorem}

\begin{proof}
We first have 
\begin{equation*}
Dx=
D
\left(
\frac{R}{Q}
\right)
=
\frac{R (Q-2 R) (2 Q-R) (Q+R)}{Q (Q-R)},
\end{equation*}
\begin{equation*}
\frac{d Q}{d x}
=
\frac{d Q}{d \tau}
\frac{d \tau}{d x} 
=
\frac{ D Q}{D x} 
=
-\frac{Q^2 \left(4 Q^4-8 Q^3 R-48 Q^2 R^2-Q^2 S+52 Q R^3+2 Q R S-20 R^4-R^2 S\right)}{12 R (Q-2 R) (Q-R) (2 Q-R) (Q+R)}
\end{equation*}
and
\begin{align*}
\frac{d^2 Q}{d x^2}
=&
\frac{d}{d \tau}
\left(
\frac{d Q}{d x}
\right)
\frac{d \tau}{d x}   \\
=&
\frac{Q^3    }{6 R^2 (Q-2 R)^2 (Q-R)^2 (2 Q-R)^2 (Q+R)^2} \times \\
&\times 
\Big(4 Q^8-20 Q^7 R+12 Q^6 R^2+104 Q^5 R^3+8 Q^4 R^4 -384 Q^3 R^5 +424 Q^2 R^6-188 Q R^7 +36 R^8\\
&
\hspace{10mm} 
-Q^6 S+5 Q^5 R S-4 Q^4 R^2 S-10 Q^3 R^3 S+20 Q^2 R^4 S-13 Q R^5 S+3 R^6 S\Big).
\end{align*}
Therefore it follows that 
\begin{align*}
&
\frac{d^2 Q}{d x^2}+
\left\{
\frac{1}{x}
-
\frac{1}{x-1}
+
\frac{1}{x+1}
+
\frac{1}{x-2}
+
\frac{1}{x-\frac12}
\right\}
\frac{d Q}{d x} \\
=&
\frac{2 Q^3 \left(Q^2-Q R+R^2\right)}{(Q-2 R) (Q-R)^2 (Q+R)} 
=
-
\frac{2 Q \left(x^2-x+1\right)}{(x-1)^2 (x+1) (2 x-1)},
\end{align*}
which proves the theorem. 
\end{proof}

\subsection{On  $R(\tau)$}

\begin{theorem}
\label{thm:R-y-linear-(0,1/3)-(0,2/3)}
{\it
For every $\tau\in\mathbb{H}^2$ set 
$
\displaystyle 
y=
\frac
{
\eta^2(\tau/2) \eta(2 \tau) \eta^9(3\tau)
}
{
\eta^3(\tau) \eta^6(3\tau/2) \eta^3(6\tau)
}. 
$ 
Then the inverse function $\tau=\tau(y)$ yields $R(y)=R(\tau(y))$ and 
$R(y)$ satisfies the following differential equation:
\begin{equation*}
\frac{d^2R}{ d y^2}
+
\left\{
\frac{1}{y}
-
\frac{1}{y-1}
+
\frac{1}{y+1}
+
\frac{1}{y-2}
+
\frac{1}{y-\frac12}
\right\}
\frac{d R}{d y}
+
\frac{2  \left(y^2-y+1\right)}{(y-1)^2 (y+1) (2 y-1)} R
=0.
\end{equation*}
}
\end{theorem}

\begin{proof}
We first have 
\begin{equation*}
Dy=
D
\left(
\frac{Q}{R}
\right)
=
-\frac{Q (Q-2 R) (2 Q-R) (Q+R)}{R (Q-R)}. 
\end{equation*}
We next obtain 
\begin{equation*}
\frac{d R}{d y}
=
\frac{d R}{d \tau}
\frac{d \tau}{ d y}
=
\frac{ D R}{D y} 
=
-\frac{R^2 \left(20 Q^4-52 Q^3 R+48 Q^2 R^2+8 Q R^3-4 R^4+Q^2 S-2 Q R S+R^2 S\right)}{12 Q (Q-2 R) (Q-R) (2 Q-R) (Q+R)}, 
\end{equation*}
and 
\begin{align*}
\frac{d^2 R}{dy^2}
=&
\frac{d}{d \tau}
\left(
\frac{d R}{d y}
\right)
\frac{d \tau}{d y}   \\
=&
\frac{R^3}{6 Q^2 (Q-2 R)^2 (Q-R)^2 (2 Q-R)^2 (Q+R)^2}  \times \\
\times&
\Big(
36 Q^8-188 Q^7 R+424 Q^6 R^2-384 Q^5 R^3+8 Q^4 R^4+104 Q^3 R^5+12 Q^2 R^6-20 Q R^7+4 R^8  \\
&\hspace{5mm} 
+3 Q^6 S-13 Q^5 R S+20 Q^4 R^2 S-10 Q^3 R^3 S-4 Q^2 R^4 S+5 Q R^5 S-R^6 S\Big). 
\end{align*}
Therefore it follows that 
\begin{align*}
&
\frac{d^2 R}{d y^2}+
\left\{
\frac{1}{y}
-
\frac{1}{y-1}
+
\frac{1}{y+1}
+
\frac{1}{y-2}
+
\frac{1}{y-\frac12}
\right\}
\frac{d R}{d y}   \\
=&
-\frac{2 R^3 \left(Q^2-Q R+R^2\right)}{(Q-R)^2 (2 Q-R) (Q+R)}
=
-\frac{2 R \left(y^2-y+1\right)}{(y-1)^2 (y+1) (2 y-1)},
\end{align*}
which proves the theorem. 
\end{proof}




\end{document}